\documentclass[11pt,reqno,british,a4paper,twoside]{amsart}

\usepackage[sc]{mathpazo}
\usepackage{amsmath,mathtools,amssymb,paralist,xspace,babel}
\usepackage{comment,multirow,hhline,array}
\usepackage{booktabs,microtype,fixltx2e}
\usepackage{caption,subcaption}
\mathtoolsset{mathic=true}
\setdefaultenum{\upshape (i)}{}{}{}

\theoremstyle{plain}
\newtheorem{theorem}{Theorem}[section]
\newtheorem{proposition}[theorem]{Proposition}
\newtheorem{lemma}[theorem]{Lemma}
\newtheorem{corollary}[theorem]{Corollary}

\newtheorem{remark}[theorem]{Remark}
\newtheorem{example}[theorem]{Example}

\numberwithin{equation}{section}
\numberwithin{table}{section}

\newcommand{\Lie}[1]{\operatorname{\rm{#1}}}
\newcommand{\lie}[1]{\operatorname{\mathfrak{#1}}}
\newcommand{\G}{\Lie{G}}
\newcommand{\LH}{\Lie{H}}
\newcommand{\LU}{\Lie{U}}
\newcommand{\GL}{\Lie{GL}}
\newcommand{\SO}{\Lie{SO}}
\newcommand{\Spin}{\Lie{Spin}}
\newcommand{\SL}{\Lie{SL}}
\newcommand{\PSL}{\Lie{PSL}}
\newcommand{\Sp}{\Lie{Sp}}
\newcommand{\SU}{\Lie{SU}}

\newcommand{\g}{\lie{g}}

\newcommand{\m}{\lie{m}}
\newcommand{\lh}{\lie{h}}
\newcommand{\gl}{\lie{gl}}
\newcommand{\so}{\lie{so}}
\newcommand{\su}{\lie{su}}
\newcommand{\lsl}{\lie{sl}}
\newcommand{\lsp}{\lie{sp}}

\newcommand{\bC}{{\mathbb C}}
\newcommand{\bH}{{\mathbb H}}
\newcommand{\bR}{{\mathbb R}}
\newcommand{\bZ}{{\mathbb Z}}
\newcommand{\pR}{\mathbb{RP}}
\newcommand{\pC}{\mathbb{CP}}

\newcommand{\bx}{\mathbf{x}}
\newcommand{\by}{\mathbf{y}}

\newcommand{\Dmod}{{\mathcal D}}

\newcommand{\Mmod}{{\mathcal M}}
\newcommand{\Nmod}{{\mathcal N}}
\newcommand{\Smod}{{\mathcal S}}

\newcolumntype{C}{>{$}c<{$}}
\newcolumntype{R}{>{$}r<{$}}

\DeclareMathOperator{\Hom}{Hom}
\DeclareMathOperator{\Ric}{Ric}
\DeclareMathOperator{\rank}{rk}
\DeclareMathOperator{\diag}{diag}

\newcommand{\hook}{{\lrcorner\,}}

\newcommand{\indoublebrackets}[1]{[\![#1]\!]}

\DeclarePairedDelimiter{\Span}{\langle}{\rangle}
\newcommand{\abs}[1]{\left\lvert#1\right\rvert}

\newcommand{\br}{\hspace{0pt}}

\newcommand{\bdash}{-\br} 

\newcommand{\D}[2]{\frac{\partial #1}{\partial #2}}

\begin{document}

\title[Invariant torsion and \( \G_2 \)-metrics]{Invariant torsion and \( \G_2 \)-metrics}

\author{Diego Conti}
\address[D.~Conti]{ Dipartimento di Matematica e Applicazioni\\ Universit\`a di Milano
  Bicocca\\ Via Cozzi 53\\ 20125 Milano\\ Italy.}
\email{diego.conti@unimib.it}

\author{Thomas Bruun Madsen}
\address[T.\,B.~Madsen]{Department of Mathematics\\ Aarhus University\\
Ny Munkegade 118, Bldg 1530\\ 8000 Aarhus\\ Denmark.}
\email{thomas.madsen@math.au.dk}

\thanks{We are grateful to Simon Salamon for many invaluable discussions and suggestions. This work was partially supported by FIRB 2012 ``Geometria differenziale e teoria geometrica delle funzioni'' and PRIN 2010-2011 ``Variet\`a reali e complesse: geometria, topologia e analisi armonica''. TBM gratefully acknowledges financial support from \textsc{villum fonden}.}

\begin{abstract}
We introduce and study a notion of invariant intrinsic torsion geometry which appears, for instance, in connection with the  Bryant-Salamon metric on the spinor bundle over \( S^3 \). This space is foliated by six-dimensional hypersurfaces, each of which carries a particular type of \( \SO(3) \)-structure; the intrinsic torsion is invariant under \( \SO(3) \). The last condition is sufficient to imply local homogeneity of such geometries, and this allows us to give a classification. We close the circle by showing that the Bryant-Salamon metric is the unique complete metric with holonomy \( \G_2 \) that arises from \( \SO(3) \)-structures with invariant intrinsic torsion.
\end{abstract}

\maketitle

\section{Introduction}
\label{sec:intro}

It was Bryant and Salamon \cite{Bryant-S:excep} who constructed the first known examples of complete, irreducible metrics with Riemannian holonomy equal to the  exceptional Lie group \( \G_2 \). All their metrics asymptotically look like that of a Riemannian cone over one of the homogeneous nearly-K\"ahler six-manifolds \( \pC(3) \), \( F_{1,2}(\bC^3) \) or \( S^3\times S^3 \). 
Whilst \( \G_2 \)-manifolds have received much attention over the past 25 years, additional examples of complete, non-compact manifolds with \( \G_2 \)-holonomy are still relatively few \cite{Brandhuber:Gaugetheoryat,Cvetic-al:G2,Bazaikin-B:completeG2}. It therefore seems sensible to return to the original examples so as to better understand what makes them so special.
The starting point of this paper is the Bryant-Salamon metric approaching the cone on \( S^3\times S^3 \). This can be viewed as a metric on the spinor bundle of \( S^3 \) which we may write as
\begin{equation*}
\Spin(4)\times_{\Sp(1)} \bH;
\end{equation*}
here the action of \( \Spin(4) \) commutes with the right action of \( \Sp(1) \) on \( \bH \). As a consequence, the spinor bundle has the structure of a cohomogeneity one manifold with principal orbit \( \Spin(4)\times \Sp(1)\slash\Sp(1) \). The isotropy representation is the sum of two copies of the adjoint representation \( \lsp(1) \); this factors through the quotient \( \SO(3)=\Sp(1)\slash \bZ_2 \) so as to give rise to \( \SO(3) \)-structures on hypersurfaces of the spinor bundle. The structure  determined by this cohomogeneity one action corresponds to the nearly-K\"ahler structure on \( S^3\times S^3 \), but other invariant \( \SO(3) \)-structures can be obtained by composing with an \( \SO(3) \)-equivariant isomorphism of \( \bR^3\oplus\bR^3 \). More generally, \( \SO(3) \)-structures come in \( \GL(2,\bR) \)-families, since \( \GL(2,\bR) \) is the centralizer of \( \SO(3) \) in \( \GL(6,\bR) \). This is at the origin of the flexibility required to construct the Bryant-Salamon metric. Another way of viewing this fact is by observing that \( \SO(3) \) preserves more than one metric on \( \bR^3\oplus \bR^3 \).

The \( \SO(3) \)-structures associated with the Bryant-Salamon metric are homogeneous. In particular, they come with an Ambrose-Singer connection, characterized by having parallel torsion and curvature. Their intrinsic torsion, which is a map from the reduced frame bundle to 
\begin{equation*}
(\bR^6)^*\otimes(\so(6)\slash\so(3)),
\end{equation*}
is therefore parallel and so determines an \( \SO(3) \)-invariant element. In this sense, \( \SO(3) \)-structures with invariant intrinsic torsion generalize the structures on \( S^3\times S^3 \) that give rise to the Bryant-Salamon metric. 

A natural source of \( \SO(3) \)-structures comes from Lie groups. Lie algebras of dimension six are characterized by the Chevalley-Eilenberg operator  \( d\colon \Lambda(\bR^6)^*\to \Lambda(\bR^6)^* \), and we may think of \( d \) as representing a Lie algebra with a fixed frame and so a natural \( \SO(3) \)-structure. If \( d \) is \( \SO(3) \)-equivariant, the associated \( \SO(3) \)-structure has invariant intrinsic torsion. It turns out that any \( \SO(3) \)-structure with invariant intrinsic torsion is locally equivalent to one obtained in this way. More precisely, non-flat \( \SO(3) \)-structures with invariant intrinsic torsion are parameterized, up to scale, by an element of \( \pR^5 \) lying in a subvariety \( \Mmod \) cut out by the condition \( d^2=0 \). Via the Segre embedding, this subvariety corresponds to \( \pR^1\times \pR^2 \). If completeness is assumed, the classification result can be made global and turns out to be closely related to the classification of naturally reductive homogeneous spaces of dimension six, recently obtained by Agricola, Ferreira and Friedrich \cite{Agricola-al:NaturallyReductive}. Indeed, to each element of \( \Mmod \) we can associate a unique naturally reductive homogeneous structure; our approach then results in a geometric way of interpreting the related part of their classification in terms of algebraic subvarieties.

In order to exploit the \( \GL(2,\bR) \)-invariance, it is convenient to identify \( \pR^n \) with the space of homogeneous degree \( n \) polynomials in two variables. The intrinsic torsion is an element of \( \bR^4 \). Projectively, it is given by a natural \( \GL(2,\bR) \)-equivariant map: 
\begin{equation}
\label{eqn:polymultiplication}
\pR^1\times \pR^2\to\pR^3,
\end{equation}
which is determined by polynomial multiplication \( (\bx ,\by )\mapsto \mathbf{xy} \). The actual connected, simply-connected Lie group that is obtained from the pair \( (\bx ,\by ) \) depends on the power of \(\bx  \) that divides \( \by  \) as well as the discriminant \( \Delta=\Delta(\by ) \).

\begin{table}[htp]
 \centering
  \begin{tabular}{CCR}
\toprule
\bx ^2\mid \by  & \Delta=0 & (0,0,0,12,13,23)\\
\bx \mid \by  & \Delta>0 & \SO(3)\ltimes\bR^3\\
\bx \nmid \by & \Delta=0 & \SO(3)\times\bR^3\\
\bx \nmid \by & \Delta>0 & \SU(2)\times\SU(2)\\
\bx \nmid \by  & \Delta<0 & \SL(2,\bC)\\
\bottomrule
\end{tabular}
\caption{The five models of invariant intrinsic torsion geometry.}
\label{table:poly}
\end{table}

Each row in Table \ref{table:poly} represents a single \( \GL(2,\bR) \)-orbit. In this way there are, up to this symmetry, precisely five non-Abelian models. However, metric properties are not invariant under \( \GL(2,\bR) \). In particular, we show that the Einstein metrics correspond to three \( \LU(1) \)-orbits in \( \pR^3 \); the geometry is \( S^3\times S^3 \) with either its nearly-K\"ahler metric or its bi-invariant metric, and the latter can appear in two ways.

The \( \GL(2,\bR) \)-symmetry is also lost  when computing the local \( \G_2 \)-metric. In fact, the Bryant-Salamon metric is obtained by integrating the Hitchin flow, which is a flow on the space of half-flat \( \SU(3) \)-structures. The half-flat condition reads as a linear condition on the \( \SO(3) \)-intrinsic torsion and determines a subset in \( \pR^1\times\pR^2 \) which can be interpreted as the blow-up of \( \pR^2 \) at a point; in the non-projective setup, which also includes the Abelian case, this gives an \( \bR^3\subset\bR^4 \). The Hitchin flow determines a flow on this \( \bR^3 \). Here the \( \GL(2,\bR) \)-symmetry breaks down to a discrete group, namely the symmetric group \( \Sigma_3 \) formed of permutations in three letters. The maximally invariant solution gives rise to the nearly-K\"ahler \( S^3\times S^3 \), whose evolution is simply the cone. The Bryant-Salamon metric appears as a \( \Sigma_2 \)-invariant solution. On the open subset \( \Delta<0 \), one can recover the metrics of \cite{Belgun:Ontheboundary}.

Strictly speaking, what we are evolving is not an \( \SU(3) \)-structure, but rather the intrinsic torsion (of an \( \SO(3) \)-structure). The actual solution, in terms of structures, is obtained by lifting the integral curves under the non-projective analogue of \eqref{eqn:polymultiplication}. This approach allows us to give a precise explanation of the ``triality'' symmetry appearing in \cite{Atiyah-W:G2metrics}; this symmetry is the reason why Brandhuber et al. \cite{Brandhuber:Gaugetheoryat} were able to construct, seemingly, different solutions whose evolution equations look exactly the same. From our point of view, there is little reason to distinguish between these cases: we obtain the same \( \G_2 \)-metric on the spinor bundle, but written in terms of three different cohomogeneity one actions of \( \SU(2)\times\SU(2) \).

By choosing appropriate one-parameter subgroups, the \( \GL(2,\bR) \)-symme\bdash{}try can also be used to generate flows that interpolate between the different Lie algebras of \( \Mmod \); such flows are sometimes referred to as Lie algebra ``contractions''. Motivated by their occurrence in the physics literature on \( \G_2 \)-metrics \cite{Gibbons-al:domainwalls,Chong-al:contr}, we classify the associated orbits of half-flat structures. Thinking of the half-flat structures as \( \bR^3\subset\bR^4 \), the union of these orbits comes in three families, each of which forms a two-plane. It turns out that only one of these planes is preserved by the Hitchin flow, meaning that the flow itself may be viewed as a contraction up to rescaling in this case. In the other two cases, contractions can be used as a way of generating initial data for Hitchin's flow equations that give rise to different solutions.

In the final result of the paper, we classify the complete holonomy \( \G_2 \)-metrics which are determined by a half-flat \( \SO(3) \)-structure with invariant intrinsic torsion. We show that the only full holonomy \( \G_2 \)-metric is the Bryant-Salamon metric. This classification result is different from, but consistent with, the uniqueness result by Karigiannis and Lotay \cite[Corollary 6.4]{Karigiannis-L:Deform_coni}. It also supplements the classification of cohomogeneity one \( \G_2 \)-structures obtained by Cleyton and Swann \cite{Cleyton-S:G2} who considered only actions of simple Lie groups.

\section{The hypersurfaces: \( \SO(3) \)-structures}
\label{sec:intrtor}

Let \( V \) denote the three-dimensional irreducible representation of \( \SO(3) \). In this paper, we are concerned with six-manifolds \( M \) whose tangent space at each point is modelled on
\begin{equation}
\label{eq:SO3tang}
 T_mM\cong T=V\oplus V.
\end{equation}
As we are free to rescale the metric on each summand \( V \), this representation theoretical definition of an \( \SO(3) \)-structure does not determine a canonical inclusion \( \SO(3) \subset \SO(6) \).
We shall address this subtlety shortly. Before doing so, however, we will explain a geometric way of achieving \eqref{eq:SO3tang}.

We start out by considering the non-degenerate \( 2 \)-form
\begin{equation}
\label{eq:defformsa}
\sigma=e^{12}+e^{34}+e^{56},
\end{equation}
together with a one-parameter family of simple \( 3 \)-forms obtained by letting \( \SO(2) \) act on the two-planes
\( \Span{e^{2i-1},e^{2i}} \):
\begin{equation*}
 \eta^{\theta}=(\cos_{\theta}e^1+\sin_{\theta}e^2)\wedge(\cos_{\theta}e^3+\sin_{\theta}e^4)\wedge(\cos_{\theta}e^5+\sin_{\theta}e^6).
\end{equation*}
If we fix three of these forms, chosen to be mutually related by rotations of \( \frac{2\pi}3 \), say, 
\begin{equation*}
\eta_0=\eta^0, \quad \eta_1=\eta^{\frac{2\pi}3}, \quad \eta_2=\eta^{-\frac{2\pi}3},
\end{equation*}
then we obtain a splitting of the form \eqref{eq:SO3tang}:

\begin{lemma}[\cite{Giovannini:SpecialStructuresAnd}]
\label{lem:def_forms}
The \( \GL(6,\bR) \)-stabilizers of the quadruplet \( (\sigma,\eta_0,\eta_1,\eta_2) \) intersect in a copy of \( \SO(3) \).
\qed
\end{lemma}

The centralizer of \( \SO(3) \) in \( \GL(6,\bR) \) is a copy of \( \GL(2,\bR) \) which contains the above rotation group \( \SO(2) \). A first clear indication that this maximal commuting subgroup should not be ignored is the fact that it can be used to remedy the ambiguity in our choice of inclusion \( \SO(3)\subset\SO(6) \). 

Having fixed such an inclusion, the identification 
\begin{equation*} 
\so(6)\cong\Lambda^2(V\oplus V) 
\end{equation*} 
allows us to consider the module \( \so(3)^\perp \) which is the orthogonal complement of \( \so(3) \) inside \( \so(6) \). This latter module is reducible with irreducible summands given by
\( \bR\oplus 2V\oplus S^2_0(V) \). Here \( S^2_0(V) \) denotes the five-dimensional irreducible representation of \( \SO(3) \) that is defined as
the kernel of the natural trace map \( S^2(V)\to \bR \). Similarly, we shall, subsequently, denote by \( S^3_0(V) \) the seven-dimensional irreducible representation, defined as the kernel of the
trace map \( S^3(V)\to V \). 

In order to exploit the additional \( \GL(2,\bR) \)-symmetry, it is useful to consider representations of the enlarged group \( \SO(3)\times\SL(2,\bR) \); restricting the attention to \( \SL(2,\bR) \) is harmless in our case, since the full symmetry group \(\SO(3)\times\GL(2,\bR) \) would result in the same decompositions into irreducible modules. The irreducible components of this enlargement have the form \( A^p\otimes B^q \) where \( A^p \), \( B^p \) denote the \( p+1 \)-dimensional irreducible representations of \( \SO(3) \) and \( \SL(2,\bR) \), respectively.

The intrinsic torsion of an \( \SO(3) \)-structure is, by definition, the projection of the torsion of any connection on the \( \SO(3) \)-structure on the cokernel of the alternating map
\begin{equation}
\label{eqn:alternatingso3}
T^*\otimes\so(3)\to \Lambda^2T^*\otimes T.
\end{equation}
This alternating map extends to an isomorphism 
\begin{equation*}
\partial\colon T^*\otimes\so(6)\to \Lambda^2T^*\otimes T,
\end{equation*}
hence the space of intrinsic torsion can be identified with \( T^*\otimes\so(3)^\perp \). It follows that any \( \SO(3) \)-structure with intrinsic torsion \( \tau \) has a unique connection with torsion \( \partial(\tau) \). Subsequently, we shall refer to this as the ``canonical connection''; its associated exterior covariant derivative operator will be denoted by  \( D \).

\begin{lemma}
\label{lemma:intrinsictorsion}
The intrinsic torsion of an \( \SO(3) \)-structure on a six-manifold belongs to the \( 72 \)-dimensional \( \SO(3) \)-module
 \begin{equation}
\label{eq:intrinsic}
T^*\otimes \so(3)^{\perp} \cong 4\bR\oplus 8V\oplus 6S^2_0(V)\oplus 2S_0^3(V).
\end{equation}

As \( \SO(3)\times\SL(2,\bR) \)-modules, this space can be expressed in the form
\begin{equation}
\label{eq:intrinsic_enlarged_grp}
B^3\oplus A^2(2B^1\oplus B^3)\oplus A^4( B^1\oplus B^3)\oplus A^6B^1.
\end{equation}
\end{lemma}

\begin{proof}
The decomposition \eqref{eq:intrinsic} follows from standard computations using the Clebsch-Gordan formula. Hence, only the expression \eqref{eq:intrinsic_enlarged_grp} needs to be explained. The intrinsic torsion takes values in the cokernel of the alternating map \eqref{eqn:alternatingso3} and this map is injective because of the inclusion \( \so(3)\subset\so(6) \). As \( \SO(3)\times\SL(2,\bR) \)-modules, we therefore have that the intrinsic torsion belongs to the quotient
\begin{equation*} 
\frac{ B^1\oplus B^3\oplus A^2(3B^1\oplus B^3)\oplus A^4(2B^1\oplus B^3)\oplus A^6B^1}{B^1\oplus A^2B^1\oplus A^4B^1} 
\end{equation*}
which immediately gives \eqref{eq:intrinsic_enlarged_grp}.
\end{proof}

The quadruplet \( (\sigma,\eta_i) \) which defines an \( \SO(3) \)-structure can be viewed as a refinement of an associated \( \SU(3) \)-structure. The latter can be realized by averaging over the forms \( \eta_i \): the \( 3 \)-form 
\begin{equation}
\label{eq:SU3_3form}
\gamma=\tfrac43\left(\eta_0+\eta_1+\eta_2\right)
\end{equation}
has \( \rm{GL}_+(6,\bR) \)-stabilizer \( \rm{SL}(3,\bC) \) which intersects with the stabilizer of \( \sigma \) in a copy of \( \SU(3) \). As \( \gamma \) has stabilizer \( \SL(3,\bC) \), it determines an almost complex structure \( J \) which can be used to define the dual \( 3 \)-form
\begin{equation*} 
\hat\gamma= J(\gamma);
\end{equation*} 
in other words \( \hat\gamma \) is the unique \( 3 \)-form such that \( \gamma+i\hat\gamma \) is a decomposable \( 3 \)-form compatible with the orientation, and \( J \) is the almost complex structure such that \( \gamma+i\hat\gamma \) is of type \( (3,0) \).

The space of \( \SU(3) \)-intrinsic torsion was studied in great details in \cite{Chiossi-S:TheIntrinsicTorsion} and can be related to the space \eqref{eq:intrinsic} via
\begin{equation*}
T^*\otimes \so(3)^{\perp} \cong T^*\otimes \su(3)^{\perp} \oplus T^*\otimes S^2_0(V).
\end{equation*}
This relation is schematically represented in Table~\ref{table:su3intrinsictorsion}.
 
\begin{table}[htp]
  \centering
	\begin{tabular}{CCC}
	\toprule
		\textrm{Component}		&		\SU(3)				&		\SO(3) \\
	\toprule
		W_1^+ \oplus W_1^-  		&		\bR\oplus \bR   		&  		\bR\oplus\bR \\
		W_2^+ \oplus W_2^- 		& 		\su(3) \oplus \su(3)		&		2V\oplus 2 S^2_0(V) \\
 	\midrule
		W_3 				& 		\indoublebrackets{S^{2,0}}	&		2\bR\oplus 2S^2_0(V) \\
	\midrule 
		W_4 				&		T 				&		V\oplus V \\
	\midrule
		W_5				&		T				&		V\oplus V \\
	\midrule 
		W_6				&						&		2V\oplus 2 S^2_0(V) \\
	\midrule 
		W_7				&				  		&		2S^3_0(V)\\
    	\bottomrule
  	\end{tabular}
  \caption{\( \SU(3) \) and \( \SO(3) \)-intrinsic torsion. To compare with \cite{Chiossi-S:TheIntrinsicTorsion}, observe that
 	\( W_1\oplus W_3\oplus W_4\cong B^3\oplus A^2B^1\oplus A^4B^1 \) and \(  W_1^- \oplus W_2^- \oplus W_5 \cong\bR\oplus 3V\oplus S^2_0(V) \).}
  \label{table:su3intrinsictorsion}
\end{table}

We have already encountered several \( \SO(3) \)-invariant forms. Amongst these, the \( 2 \)-form \( \sigma \) and its square \( \sigma^2 \) exhaust the invariant forms of degree two and four: 
\begin{equation*} 
\Lambda^2 T^*\cong \bR\oplus 3V\oplus S^2_0(V)\cong\bR\oplus A^2B^2\oplus A^4. 
\end{equation*}
We have also introduced five invariant \( 3 \)-forms. These are not all independent as is evident from the decomposition: 
\begin{equation*} 
\Lambda^3 T^*\cong 4\bR\oplus 2V\oplus 2S^2_0(V)\cong B^3\oplus A^2B^1 \oplus A^4B^1. 
\end{equation*}
A convenient basis consists of the four forms \( \eta_i, \hat\gamma \).
 
The exterior derivatives of the invariant forms detect part of the intrinsic torsion. In fact, they determine all of it apart from the \( 14 \)-dimensional module \( A^6B^1 \) of \eqref{eq:intrinsic_enlarged_grp}.

\begin{theorem}
\label{thm:SO3formsdetermineIT}
Given an \( \SO(3) \)-structure \( (\sigma,\eta_i) \), the associated invariant forms determine part of the intrinsic torsion as follows:
\vspace{.5mm}
\begin{compactenum}
\item \label{thm1:itm1} \( d\sigma \) determines the subspace \( 3\bR\oplus 2V\oplus 2S^2_0(V) \);
\item \label{thm1:itm2} each \( d\eta_i \) determines a subspace \(\bR\oplus 2V\oplus S^2_0(V) \), and the three resulting subspaces are in direct sum;
\item \label{thm1:itm3} \( d\hat\gamma \) determines the subspace \(\bR\oplus 3V\oplus S^2_0(V) \).
\end{compactenum}
\vspace{.5mm}
In terms of \( \SO(3)\times\SL(2,\bR) \)-modules, this means that
 \( d\sigma \) determines the subspace 
\begin{equation}
\label{eq:int_2form}
B^3\oplus A^2B^1\oplus A^4B^1, 
\end{equation}
and the \(d\eta_i\), \(d\hat\gamma\) altogether determine the modules
\begin{equation}
\label{eq:int_3forms}
B^3\oplus A^2B^1\oplus A^2B^3\oplus A^4B^3.
\end{equation}

In particular, the only intrinsic torsion component not determined by the exterior derivatives of invariant forms is the \( 14 \)-dimensional module \( A^6B^1 \).
\end{theorem}
 
\begin{proof}
We consider the alternating map
\begin{equation}
\label{eqn:alternating}
T^*\otimes\gl(6,\bR)\to \Lambda^2T^*\otimes T,
\end{equation}
and assume its restriction to the direct sum \( W\oplus T^*\otimes\so(3) \) is an isomorphism; any such subspace \( W \) can be identified with the space of intrinsic torsion. We may assume \( W \) is invariant under \( \SO(3)\times\GL(2,\bR) \).

Given any \( \SO(3) \)-invariant form \( \alpha \) in \( \Lambda^kT^* \), the infinitesimal action of \( \gl(6,\bR) \) determines an equivariant linear map 
\begin{equation*}
l_\alpha\colon  T^*\otimes \gl(6,\bR)\to \Lambda^{k+1}T^*. 
\end{equation*}
Denote by \( W^\alpha \) a maximal submodule on which \( l_\alpha \) is injective. As \( l_\alpha \) is zero on both \( T^*\otimes\so(3) \) and the kernel of \eqref{eqn:alternating}, we can assume \( W^\alpha\subset W \). By construction, \( W^\alpha \) represents the component of \( W \) that is determined by \( d\alpha \). 

Straightforward computations verify that \( l_\sigma \) is surjective. Since it is also \( \SO(3)\times\SL(2,\bR) \)-equivariant, \(W^{ \sigma} \) is an \( \SO(3)\times\SL(2,\bR) \)-module isomorphic to \( \Lambda^3T^*=B^3\oplus A^2B^1 \oplus A^4B^1 \). This proves \eqref{thm1:itm1} and \eqref{eq:int_2form}.

The maps \(l_\gamma \) and \( l_{\hat\gamma} \) are also surjective, proving \eqref{thm1:itm3}, whereas the image of \( l_{\eta_i} \) is isomorphic to \( \bR\oplus 2V\oplus S^2_0(V) \). For example, we  find that the image of \( l_{\eta_0} \), inside \( \Lambda^4T^* \), is the orthogonal complement of
\begin{equation*} 
\Span{e^{1246},e^{2346},e^{2456}}\cong V. 
\end{equation*}
Similar computations hold for \( \eta_1 \) and \( \eta_2 \). This proves \eqref{thm1:itm2}.

The space of \( \SO(3) \)-invariant \( 3 \)-forms is the \( \SO(3)\times\SL(2,\bR) \)-module \( \Span{\hat\gamma,\eta_0, \eta_1,\eta_2}\cong B^3 \). It follows that the part of intrinsic torsion \(\sum_i W^{\eta_i} + W^{\hat\gamma} \), determined by these forms, is a submodule of
\begin{equation*} 
\Hom( B^3, \Lambda^4T^*)\cong B^3\oplus A^2(B^1\oplus B^3\oplus B^5)\oplus A^4B^3. 
\end{equation*}
In fact, the module \( A^2B^5 \) does not appear in the decomposition \eqref{eq:intrinsic_enlarged_grp}, and  \( \sum_i W^{\eta_i} + W^{\hat\gamma} \)  contains at least one copy of \( \bR \) and \( S^2_0V \) as an \( \SO(3) \)-module. It must therefore contain 
\( B^3\oplus A^4B^3 \). In order to prove \eqref{eq:int_3forms}, it then suffices to prove that it also contains \( A^2B^1\oplus A^2B^3 \).

In order to prove this, we consider the symmetric group \( \Sigma_3 \) which can be viewed as a subgroup of \( \GL(2,\bR) \) via
\begin{equation}
\label{eq:sigma3ingl2}
(12)\mapsto \begin{pmatrix} -\frac12 & \frac{\sqrt3}2 \\ \frac{\sqrt3}2 & \frac12\end{pmatrix},\quad (23)\mapsto  \begin{pmatrix} 1 & 0 \\ 0 & -1\end{pmatrix}.
\end{equation}
Its irreducible representations are the trivial and alternating representations \( \bR \) and \( U \) (both one-dimensional) and the standard representation \( S \) (which has dimension two). The latter is induced by the inclusion \eqref{eq:sigma3ingl2}.

As \( \SO(3)\times\Sigma_3 \)-modules, we have:
\begin{equation*}
T\cong A^2S, \quad \Lambda^2T^*\cong A^2(S+\bR)+(A^0+A^4)U.
\end{equation*}
Since this is is an orthogonal group, we have \( \Lambda^2T^*\cong\so(6) \). Moreover, \( \Sigma_3 \) acts trivially on \( \so(3) \) because \( \SO(3) \) commutes with \( \GL(2,\bR) \). It follows that we may write
\begin{equation*}
T^*\otimes \so(3)^\perp=A^6S + A^4(2S+U+\bR)+A^2(3S+U+\bR)+S+U+\bR.
\end{equation*}
Now, using \( \Lambda^2T^*\cong \Lambda^4T^*\otimes U \) and \( \Span{\hat\gamma}\cong U \), \( \Span{\eta_i}\cong S+\bR \), we find that
\begin{equation*}
\begin{gathered}
W^{\hat\gamma}= A^2(S+\bR)+(A^0+A^4)U,\\
\sum_i W^{\eta_i}\subset A^2(3S+U+\bR)+(A^0+A^4)(S+\bR).
\end{gathered}
\end{equation*}

The \( \SO(3) \)-modules isomorphic to \( S^2_0(V) \) are identified by the \( \SO(3)\times\Sigma_3 \)-equivariant maps, \( g_1,g_2,g_3\colon T^*\to T^*\otimes\so(3)^\perp \),
\begin{equation*}
\begin{gathered}
g_1(v)= (v\hook\sigma)\otimes\sigma,\quad g_2( v)= e^a\otimes [v\wedge e^a]_{\so(3)^\perp}, \\
g_3( v)= e^a\otimes (v\wedge e^a-v\hook\sigma\wedge e_a\hook\sigma),
\end{gathered}
\end{equation*}
and the two \( \SO(3) \)-equivariant maps \( g_4,g_5\colon V\to T^*\otimes\so(3)^\perp \), given by:
\begin{equation*}
\begin{gathered}
g_4(v) =e_a\hook(v+v\hook\sigma)\hook(e^{135}+e^{246})\otimes e_a\hook\hat\gamma,\\
  g_5(v)= e_a\hook v\hook e^{135}\otimes (e_a\hook\sigma)\hook\hat\gamma
  +e_a\hook\sigma\otimes (e_a\hook v\hook e^{135})\hook\hat\gamma.
\end{gathered}
\end{equation*}
These images are isomorphic to \( A^2U \) and \( A^2 \), respectively.

Setting 
\begin{equation}
 \label{eqn:tauinV}
\tau= \sum_{i=1}^5 g_i(v_i) + \sum_{j=1}^3 g_j(u_j e_2),
\end{equation}
we compute
\begin{equation*} 
l_{\eta_0}(\tau) =   {(2 v_4-3 u_2-u_1-2 u_3)} e^{1235}-   {(v_1+v_3+v_5)}{(e^{1236}+e^{1245})}.
\end{equation*}
This shows that each \( W^{\eta_i} \) contains an \( \SO(3) \)-module isomorphic to \( 2V \) whose projections to \( 3A^2S \), \( A^2U \) and \( A^2 \) have maximal rank, hence
\begin{equation}
 \label{eqn:Wetai}
W^{\eta_0}+W^{\eta_1}+W^{\eta_2}\supset A^2(2S+U+\bR)=A^2(B^1+B^3),
\end{equation}
which was what remained to be proved. 
\end{proof}

It is clear from the above that \( \SO(3) \)
is the largest subgroup that acts trivially on the space of invariant forms, but the closedness of the
invariant forms fails to ensure the vanishing of the intrinsic torsion: in the language of \cite{Bryant:Calibrated}, \( \SO(3) \)-structures are admissible but not strongly admissible.

Despite this gap between ``closedness'' and ``parallelness'', \( \SO(3) \)-struc\bdash{}tures enjoy a rather remarkable feature which ties the curvature with the intrinsic torsion. The following result is, in some sense, an analogue of Bonan's results about Ricci-flatness of metrics with exceptional holonomy \cite{Bonan:Surdesvarietes}. In the setting of \( \SO(3) \)-structures, however, the curvature information encoded in the intrinsic torsion is not limited to the Ricci curvature. The following result implies that a torsion-free \( \SO(3) \)-structure is, in fact, flat.

\begin{proposition}
\label{Prop:curv_int_tor}
There are two \( \SO(3) \)-equivariant linear maps \( R^{\so(6)} \), \( R^{\so(3)} \)
such that whenever \( P \) is an \( \SO(3) \)-structure with intrinsic torsion \( \tau \), the composition 
\begin{equation*}
P\xrightarrow{(\tau\otimes\tau,D\tau)} S^2(T^*\otimes\so(3)^\perp) \oplus \Lambda^2T^*\otimes\so(3)^\perp\xrightarrow{R^{\so(6)}} \Lambda^2T^*\otimes\so(6)
\end{equation*}
is the curvature of the Levi-Civita connection, and 
\begin{equation*}
P\xrightarrow{(\tau\otimes\tau,D\tau)} S^2(T^*\otimes\so(3)^\perp) \oplus \Lambda^2T^*\otimes\so(3)^\perp\xrightarrow{R^{\so(3)}} \Lambda^2T^*\otimes\so(3)
\end{equation*}
is the curvature of the canonical connection.
\end{proposition}

\begin{proof}
Let \( \omega^{\so(6)} \) denote the restriction of the Levi-Civita connection one-form to the (reduced frame bundle) \( P \). The canonical connection form \( \omega \) is the orthogonal projection of \( \omega^{\so(6)} \) onto \( \so(3) \), and the intrinsic torsion \( \tau \) is the difference \( \omega-\omega^{\so(6)} \). We can express the curvature of \( \omega^{\so(6)} \) as
\begin{equation*} 
\Omega^{\so(6)}=D\omega-D\tau+\tfrac12[\tau,\tau]_{\so(3)^\perp}+\tfrac12[\tau,\tau]_{\so(3)}.
\end{equation*}

The Bianchi identity \( \Omega^{\so(6)}\wedge\theta=0 \) determines a subspace \( \mathcal R \) in 
\(  \Lambda^2\otimes \so(6) \) given as the kernel of the natural map
\begin{equation*}
S^2(\Lambda^2)\to\Lambda^4 .
\end{equation*}
It is easy to check that \( \mathcal{R} \) intersects trivially with \( S^2(\so(3)) \). It follows that the projection 
\begin{equation}
\label{eqn:pperp}
p^\perp\colon  \mathcal R\to \Lambda^2\otimes \so(3)^\perp 
\end{equation}
is  injective. If \( s \) is a left inverse of \( p^\perp \), we have 
\begin{equation*} 
\Omega^{\so(6)} = s(-D\tau+\tfrac12[\tau,\tau]_{\so(3)^\perp}), 
\end{equation*}
which gives the map \( R^{\so(6)} \). Similarly, \( R^{\so(3)} \) is determined by
\begin{equation*}
D\omega =(\Omega^{\so(6)})_{\so(3)}-\tfrac12[\tau,\tau]_{\so(3)}. 
\qedhere
\end{equation*}
\end{proof}

\begin{remark}
 \label{remark:pperpinjective}
Proposition \ref{Prop:curv_int_tor} implies that when \( \tau=0 \) the curvature is also zero; this is a consequence of the fact that the projection \( \mathcal R\to \Lambda^2\otimes \so(3)^\perp \ \) is injective, and holds more generally for \( \G \)-structures where \( \G \) acts reducibly on \( T \) and faithfully on each irreducible component (see also the work of Cleyton and Swann \cite{Cleyon-S:Einstein}).
\end{remark}

As \( \SO(3) \) is not strongly admissible, we cannot express the curvature solely in terms of the exterior derivative of the defining forms. A way to circumvent this fact is achieved by using the inclusion of \( \SO(3) \) in \( \SU(3) \). Since the latter structure group is strongly admissible, the Ricci curvature can be expressed in terms of the exterior derivatives of the invariant forms \( \sigma \) and \( \gamma+i\hat{\gamma} \). As explained in \cite{Bedulli-B:TheRicciTensor} this, in particular, applies to \emph{half-flat \( \mathit{SU(3)} \)-structures}, meaning those which have 
\begin{equation}
\label{eq:SU3_halfflat}
d\sigma^2=0, \quad d\gamma=0.
\end{equation}
In terms of the decomposition of \cite{Chiossi-S:TheIntrinsicTorsion}, these are the structures whose intrinsic torsion belongs to the \( 21 \)-dimensional submodule \( W_1^-\oplus W_2^-\oplus W_3 \).

\section{Invariant torsion}
\label{sec:inv_intr}

A natural generalization of torsion-free \( \SO(3) \)-structures are those which have \emph{invariant intrinsic torsion} \( \tau \), meaning 
\begin{equation*} 
\tau\in(T^*\otimes\so(3)^\perp)^{\SO(3)}\cong 4\bR. 
\end{equation*}
 
Whilst the invariance refers to the action of \( \SO(3) \), a significant role is played by the commuting subgroup \( \SL(2,\bR)\subset\GL(6,\bR) \). With respect to this group, the module \( (T^*\otimes\so(3)^\perp)^{\SO(3)} \) appears as \( B^3 \). Explicitly, we can choose coordinates \( u_1,u_2 \) on \( \bR^2 \) such that the map 
\begin{equation*} 
(\bR^2)^*\to T^*, \quad u_1\mapsto e^1, u_2\mapsto e^2
\end{equation*}
is \( \SL(2,\bR) \)-equivariant, and then we may think of \( B^k \) as the space \(  \bR_k[u_1,u_2] \) of homogeneous polynomials of degree \( k \) in \( u_1\) and \(u_2 \). This space is naturally mapped to a subspace of the tensorial algebra over \( B^1 \) via 
\begin{equation*}
u_{i_1}\dotsm u_{i_k}\mapsto \sum_{\alpha\in\Sigma_k} \tfrac1{k!}u_{i_{\alpha_1}}\otimes \dotsb \otimes u_{i_{\alpha_k}}.
\end{equation*} 

The intrinsic torsion, strictly speaking, takes values in a quotient of \( \Lambda^2T^*\otimes T \). The subspace of this quotient being fixed by \( \SO(3) \) is identified via the following:

\begin{lemma}
\label{lemma:B3invtor}
There is an \( \SO(3)\times\SL(2,\bR) \)-equivariant map,
\begin{equation*} 
B^3\oplus B^1\to \Lambda^2T^*\otimes T\cong  \Hom(T^*,\Lambda^2T^*),  
\end{equation*}
mapping
\begin{equation*}
\begin{gathered}
(\lambda,\mu)=\bigl(\lambda_1 u_1^3+\lambda_2 u_1^2 u_2+\lambda_3 u_2^2 u_1+\lambda_4 u_2^3,\mu_1u_1+\mu_2u_2\bigr)
\end{gathered}
\end{equation*}
to  \(\kappa_{\lambda,\mu}\) satisfying
\begin{equation*}
\begin{gathered}
  \kappa_{\lambda,\mu}(e^1)= (\mu_1-\tfrac13\lambda_2) e^{35}+(\tfrac12\mu_2-\tfrac13\lambda_3)(e^{36}+e^{45})-\lambda_4 e^{46}, \\
  \kappa_{\lambda,\mu}(e^2) = (\tfrac13\lambda_2+\tfrac12\mu_1)(e^{36}+e^{45}) +\lambda_1 e^{35} +(\mu_2+\tfrac13\lambda_3)e^{46}.
\end{gathered}
\end{equation*}
\end{lemma}

\begin{proof}
There is a natural inclusion \( B^3=S^3(B^1)\subset B^1\otimes S^2(B^1) \) given by 
\begin{equation*}
u_1^3\mapsto u_1\otimes u_1^2, \quad u_1^2u_2\mapsto\tfrac23 u_1\otimes u_1u_2 + \tfrac13u_2\otimes u_1^2,
\end{equation*}
and an equivariant map \( B^1\to B^1\otimes S^2(B^1) \),
\begin{equation*}
u_1\mapsto u_1\otimes u_1 u_2-u_2\otimes u_1^2, \quad u_2\mapsto -u_2\otimes u_1 u_2+u_1\otimes u_2^2. 
\end{equation*}
The statement is then obtained by considering the map
\begin{equation*}
\begin{gathered}
B^1\otimes B^1\otimes B^1\to T\otimes \Lambda^2T^*,\\ 
u_i\otimes u_j\otimes u_k \mapsto (u_ie^1\hook\sigma) \otimes u_je^3\wedge u_ke^5
+ (u_ie^3\hook\sigma) \otimes u_je^5\wedge u_ke^1\\
+ (u_ie^5\hook\sigma) \otimes u_je^1\wedge u_ke^3.
\end{gathered}
\end{equation*}
\end{proof}

It follows that we can identify the intrinsic torsion \( \tau \) with a quadruplet of functions \( \lambda_1,\lambda_2,\lambda_3, \lambda_4\in C^\infty(M) \).
These then govern a differential complex, which is obtained by restriction of the exterior derivative to the span of the set of invariant forms; it is similar to the situation considered in \cite{Chiossi-M:SO3Structures}. This complex is completely determined by
\begin{equation}
\label{eq:inv_tor_ext_der}
\begin{gathered}
d\eta_0 = -  \tfrac12\lambda_4 \sigma^2,\\
d\eta_1 =  \tfrac1{16}(3 \sqrt{3} \lambda_1+3 \lambda_2+  \sqrt{3} \lambda_3+\lambda_4)\sigma^2,\\
d\eta_2 =  -\tfrac1{16}(3  \sqrt{3} \lambda_1-3 \lambda_2+  \sqrt{3} \lambda_3-\lambda_4)\sigma^2,\\
d\hat\gamma = \tfrac12(\lambda_3-\lambda_1)\sigma^2,\\
d\sigma =\tfrac34(\lambda_1-\lambda_3)\gamma +\tfrac34(\lambda_2-\lambda_4) \hat\gamma  + \beta,
\end{gathered}
\end{equation}
where
\begin{multline*}
\beta=\tfrac{1}{4}  (\lambda_2+3\lambda_4) (e^{235}+e^{145}+e^{136}+3e^{246})\\
 +\tfrac{1}{4}  (3\lambda_1+\lambda_3)(e^{245}+e^{146}+e^{236}+3e^{135}).
\end{multline*}

\begin{remark}
Expressing \( d\sigma \) this way allows us to easily determine the components of the \( \SU(3) \)-intrinsic torsion:
\begin{equation}
\label{eq:SU3_intr_tor}
W_1^+ = \tfrac12(\lambda_2-\lambda_4), \quad W_1^-=\tfrac12(\lambda_3-\lambda_1), \quad W_3= \beta.
\end{equation}
\end{remark}

For \( \SO(3) \)-structures, it turns out that invariant intrinsic torsion is the same as constant intrinsic torsion, as in \cite{Conti-M:HarmonicStructuresAnd}:

\begin{lemma}
\label{lemma:ticonstant}
On a connected manifold, any \( \SO(3) \)-structure with invariant intrinsic torsion has constant torsion, i.e., the functions \( \lambda_i \) are constant.
\end{lemma}

\begin{proof}
Since there are no invariant \( 5 \)-forms, \( \sigma^2 \) is necessarily closed. Applying \( d \) to the equations~\eqref{eq:inv_tor_ext_der}, we find that
\( d\lambda_i\wedge \sigma^2 \) vanishes, for \( i=1,2,3,4 \). As \( \sigma \) is non-degenerate, we conclude \( d\lambda_i=0 \), and the statement follows.
\end{proof}

If we view the intrinsic torsion as taking values in \( T^*\otimes\so(3)^\perp \), there is a cost, since this is not an \( \SL(2,\bR) \)-module. 
Nevertheless, this way of viewing things will be useful subsequently.
\begin{lemma}
\label{lemma:tauS}
The intrinsic torsion \( \tau_\lambda \) can be written as
\begin{equation*} 
\kappa=\partial(\tau_\lambda) \mod \partial(T^*\otimes \so(3)), 
\end{equation*}
where \( \tau_\lambda\in T^*\otimes \so(3)^\perp \) is given by:
\begin{equation*}
\begin{gathered}
\tfrac14 (\lambda_1+ \lambda_3)(e^2\otimes e^{46}-e^2\otimes e^{35}-e^4\otimes e^{26}+e^4\otimes e^{15}+e^6\otimes e^{24}-e^6\otimes e^{13})\\
-\tfrac12 \lambda_4 (e^2\otimes e^{45}+e^2\otimes e^{36}-e^4\otimes e^{25}-e^4\otimes e^{16}+e^6\otimes e^{14}+e^6\otimes e^{23})\\
+\tfrac14 (\lambda_2+ \lambda_4)(e^1\otimes e^{46}-e^1\otimes e^{35}-e^3\otimes e^{26}+e^3\otimes e^{15}+e^5\otimes e^{24}-e^5\otimes e^{13})\\
+\tfrac12  \lambda_1(e^1\otimes e^{45}+e^1\otimes e^{36}-e^3\otimes e^{25}-e^3\otimes e^{16}+e^5\otimes e^{14}+e^5\otimes e^{23}).
\end{gathered}
\end{equation*}
\end{lemma}

\begin{proof}
The alternating map \( T^*\otimes\so(3)^\perp\to\Lambda^2T^*\otimes T \) induces an isomorphism
\( \bR^4\rightarrow B^3\oplus B^1\rightarrow B^3 \); here the second map is the projection, and the inclusion \( B^3\oplus B^1\subset\Lambda^2T^*\otimes T \) is the one given in Lemma~\ref{lemma:B3invtor}. Explicit computation of the inverse then gives the stated formula.
\end{proof}

Combining the above observations with Proposition \ref{Prop:curv_int_tor}, we see that invariant intrinsic torsion structures are locally uniquely determined by the torsion. 
\begin{theorem}
\label{thm:citclassifies}
Two \( \SO(3) \)-structures with invariant intrinsic torsion are locally equivalent if and only if they have the same intrinsic torsion. If, in addition, the underlying manifolds are complete, connected and simply-connected, then the structures are globally equivalent.
\end{theorem}

\begin{proof}
Let \( P\to M \) be an \( \SO(3) \)-structure with invariant intrinsic torsion.
By Lemma~\ref{lemma:ticonstant}, the intrinsic torsion is constant which means the function \( \tau\colon P\to T^*\otimes\so(3)^\perp \), defined in Lemma~\ref{lemma:tauS}, is constant and so parallel. Thus, as a tensorial form \( \tau\in \Omega^1(P,\so(3)^\perp) \), we have 
\begin{equation*}
D\tau = \Theta\hook \tau, 
\end{equation*}
where \( \Theta\in\Omega^2(P,T) \) denotes the torsion, and \( \hook \) represents the contraction
\begin{equation*}
(\Lambda^2T^*\otimes T)\otimes (T^*\otimes \so(3)^\perp) \to \Lambda^2T^*\otimes \so(3)^\perp.
\end{equation*}

By Proposition~\ref{Prop:curv_int_tor}, we have
\( \Omega^{\so(3)} = R^{\so(3)}(\tau\otimes\tau, D\tau) \),
showing that the curvature of the canonical connection is completely determined by the intrinsic torsion. The curvature is therefore constant as a map
\begin{equation*}
P\to \bR^3\subset\Lambda^2\otimes\so(3). 
\end{equation*}
This implies that \( \Omega^{\so(3)} \) is parallel, and the same applies to \( \Theta \). 

Consider now two such structures \( P\to M \) and \( P'\to M' \) satisfying \( \tau(u)=\tau'(u') \). Then, by \cite[Theorem VI.7.4]{KobayashiNomizu:vol1}, there is a local affine isomorphism mapping \( u \) to \( u' \). Since it is affine with respect to an \( \SO(3) \) connection, it maps \( P \) into \( P' \), thereby giving a local equivalence.

If \( M \) and \( M' \) are connected, simply-connected and complete, then \cite[Theorem VI.7.8]{KobayashiNomizu:vol1} implies that the equivalence can be extended globally.
\end{proof}

\begin{remark}
In the  constant intrinsic torsion setting, the statement of Theorem~\ref{thm:citclassifies} holds under the conditions mentioned in Remark~\ref{remark:pperpinjective}.
\end{remark}

By the proof of the above theorem, the canonical connection \( \nabla \) is an \emph{Ambrose-Singer connection}, meaning its torsion and curvature tensors are parallel.  In other terms  \cite{Ambrose-S:OnHomogeneousRimeannian}, \( (M,\nabla) \) is locally homogeneous. The following result, whose proof we shall defer, implies that \( M \) is, in fact, locally isometric to a homogeneous space.

\begin{corollary}
\label{cor:locallyhomogeneous}
A Riemannian manifold with an \( \SO(3) \)-structure with invariant intrinsic torsion is locally isometric to a Lie group with a left-invariant metric. Moreover any element of the subspace \( 4\bR \) of the space of intrinsic torsion can be realized in this way.
\end{corollary}

Following \cite{Tricerri-V:homstruc}, there are eight classes of homogeneous structures which are defined according to the action of the orthogonal group on \( T^*\otimes \Lambda^2T^* \); we have \( T^*\otimes \Lambda^2T^*\cong \mathcal J_1\oplus\mathcal J_2\oplus\mathcal J_3 \) with \( \mathcal J_1\cong T^* \) and \( \mathcal J_3\cong \Lambda^3T^* \). From Lemma~\ref{lemma:tauS}, we see that \((M,\nabla) \) can either have mixed type \( \mathcal J_2\oplus \mathcal J_3 \) or pure type \( \mathcal J_3 \), also referred to as \emph{naturally reductive}; the latter case happens precisely when  \( \lambda_2+3\lambda_4=0=3\lambda_1+\lambda_3 \).

Before addressing the proof of Corollary \ref{cor:locallyhomogeneous}, we should like to emphasize that the interesting invariant intrinsic torsion \( \SO(3) \)-structures are precisely those which are not symplectic, in the following sense:

\begin{proposition}
\label{prop:locflat}
An \( \SO(3) \)-structure with invariant intrinsic torsion is symplectic if and only if it is locally equivalent to \( T^*\bR^3 \) with its flat \( \SO(3) \)-structure.
\end{proposition}

\begin{proof}
The space \( T^*\bR^3 \) can be equipped with a natural \( \SO(3) \)-structure, compatible with the canonical symplectic form. This structure has vanishing intrinsic torsion and is, by Theorem \ref{thm:citclassifies}, locally uniquely determined by this condition. 

In general, we note that the exterior derivative of \( \sigma \) completely determines the component of the intrinsic torsion isomorphic to \( 4\bR \). Therefore, 
a symplectic manifold with invariant \( \SO(3) \)-intrinsic torsion must be torsion-free. The statement of the proposition now follows by local uniqueness.
\end{proof}

\subsection{The invariant torsion variety \( \Mmod \)}

Let \( e^1,\ldots, e^6 \) be an \( \SO(3) \) adapted basis of the dual of the Lie algebra \( \g^* \) of \( \G \), or, correspondingly, think of this as a coframe on \( \G \) that determines a left-invariant \( \SO(3) \)-structure. The torsion of the flat connection can then be expressed as
\begin{equation}
\label{eq:tor_flat_conn} 
\sum_ide^i\otimes e_i, 
\end{equation} 
and the intrinsic torsion of the structure is invariant precisely when \eqref{eq:tor_flat_conn} is an element of \( B^3\oplus\partial(T^*\otimes\so(3)) \). 

We shall now investigate the Lie algebras that are determined by elements \( \kappa_{\lambda,\mu} \) as in Lemma~\ref{lemma:B3invtor}. By the above, these have invariant intrinsic torsion.
In order to appropriately parameterize this family of Lie algebras, we introduce the following algebraic variety:
\begin{equation*} 
\Mmod = \{[\lambda_1:\lambda_2:\lambda_3:\lambda_4:\mu_1:\mu_2]\in\pR^5\mid \rank Q_{\lambda,\mu}=1\}, 
\end{equation*}
where
\begin{equation*} 
Q_{\lambda,\mu}=\begin{pmatrix} -\tfrac23\lambda_2+\tfrac12\mu_1 & -\frac13\lambda_3+\tfrac12\mu_2 & \lambda_1 \\ 
-\tfrac23\lambda_3-\tfrac12\mu_2 &- \lambda_4 & \frac13\lambda_2+\frac12\mu_1\end{pmatrix}.
\end{equation*}

\begin{lemma}
\label{lemma:b1tensorb3}
The isomorphism of \( \SL(2,\bR) \)-modules \( B^1\otimes B^2\cong B^3\oplus B^1 \),
\begin{multline*}
(x_1u_1+x_2u_2)\cdot(y_1u_1^2+y_2u_1u_2+y_3u_2^2)\\
\mapsto
((x_1y_1,x_1y_2+x_2y_1, x_1y_3+x_2y_2, x_2y_3 ),(\tfrac23(2x_2y_1-x_1y_2) ,\tfrac23(x_2y_2-2x_1y_3)))
\end{multline*}
induces a \( \PSL(2,\bR) \)-equivariant isomorphism \( \pR^1\times \pR^2\to\Mmod \).
\end{lemma}

\begin{proof}
The fact that the stated map is an isomorphism
\( B^1\otimes B^2\cong B^3\oplus B^1 \) 
follows from the computation:
\begin{multline*}
(x_1u_1+x_2u_2)\cdot (y_1u_1^2+y_2u_1u_2 + y_3u_2^2) \\
 =x_1y_1u_1^3 + (x_1y_2+x_2y_1)u_1^2 u_2 + (x_2y_2+x_1y_3)u_2^2 u_1 + x_2y_3u_2^3\\
+\tfrac13(x_1y_2-2x_2y_1) (u_1\cdot u_1 u_2-u_2\cdot u_1^2)\\
+\tfrac13(x_2y_2-2x_1y_3)(u_2\cdot u_1 u_2-u_1\cdot u_2^2);
\end{multline*}
the coefficient of the second summand has been chosen so as to obtain \( \Mmod \) as the image of the induced projective map  \( \pR^1\times\pR^2\to\pR^5 \). Up to a change of coordinates, this latter map is the Segre embedding and therefore an isomorphism.
\end{proof}

Prompted by Lemma \ref{lemma:b1tensorb3}, we will represent  elements of  \( \Mmod \) by formal products of polynomials:
\begin{equation*}
(x_1u_1+x_2u_2)\cdot(y_1u_1^2+y_2u_1u_2+y_3u_2^2).
\end{equation*}
\begin{proposition}
\label{prop:inv_tor_flat_conn}
Modulo rescaling of the metric each point of \( \Mmod\cong\pR^1\times\pR^2 \) corresponds to an invariant intrinsic torsion \( \SO(3) \)-structure on a non-Abelian Lie algebra \( \g \). The structural equations of \( \g \), in terms of the adapted basis \( e^i \) of \( \g^* \), are given by
\begin{equation}
\begin{aligned}
\label{eq:lieinvtor_proj}
de^1 &=( x_2 y_1- x_1 y_2)e^{35}- x_1y_3(e^{36}+e^{45})- x_2 y_3 e^{46},\\
de^3 &= -( x_2 y_1- x_1 y_2)e^{15}+ x_1 y_3(e^{16}+e^{25})+ x_2 y_3 e^{26},\\
de^5 &= ( x_2 y_1- x_1 y_2)e^{13}- x_1y_3(e^{14}+e^{23})- x_2 y_3 e^{24},\\
de^2 &=  x_1y_1 e^{35} -( x_1 y_3- y_2 x_2)e^{46}+ x_2y_1(e^{36}+e^{45}), \\
de^4 &= - x_1 y_1e^{15} +( x_1 y_3- y_2 x_2)e^{26}- x_2 y_1(e^{16}+e^{25}),\\
de^6 &=  x_1 y_1e^{13} - ( x_1 y_3- y_2 x_2)e^{24}+ x_2 y_1(e^{14}+e^{23}). 
\end{aligned}
\end{equation}

In addition, the associated flat connection has invariant torsion, and the intrinsic torsion is represented by the expanded product
\begin{equation*}
x_1y_1u_1^3+(x_1y_2+x_2y_1)u_1^2u_2 + (x_1y_3+x_2y_2)u_1u_2^2+x_2y_3u_2^3.
\end{equation*}

In particular the intrinsic torsion determines a surjective \( \PSL(2,\bR) \)-equivariant map \( \pR^1\times \pR^2\to \pR^3 \).
\end{proposition}

\begin{proof}
After relabelling we find that an element \( \kappa_{\lambda,\mu} \), given as in Lemma~\ref{lemma:B3invtor}, corresponds to the structural equations
\begin{equation}
\begin{gathered}
\label{eq:lieinvtor}
de^1=ae^{35}+be^{46}+c(e^{36}+e^{45}),\quad  de^2=qe^{35}+pe^{46}+r(e^{36}+e^{45}),
\end{gathered}
\end{equation}
and so forth, where
\begin{equation*}
\begin{gathered}
a=-\tfrac13\lambda_2+\mu_1,\quad b=-\lambda_4,\quad c=-\tfrac13\lambda_3+\tfrac12\mu_2,\\
q=\lambda_1,\quad p=\tfrac13\lambda_3+\mu_2,\quad r=\tfrac13\lambda_2+\tfrac12\mu_1.
\end{gathered}
\end{equation*}
Note that rescaling of the metric amounts to the change \( e^i\mapsto z \tilde{e}^i \), for a non-zero constant \( z \).
Then, letting \( \tilde{a}=a\slash z,\ldots,\tilde{r}=r\slash z \), we have that \( [\tilde{a}:\ldots:\tilde{r}]=[a:\ldots :r] \) in \( \pR^5 \) and \( d\tilde{e}^1=\tilde{a}\tilde{e}^{35}+\tilde{b}\tilde{e}^{46}+\tilde{c}(\tilde{e}^{36}+\tilde{e}^{45}) \). Also observe that \eqref{eq:lieinvtor} becomes \eqref{eq:lieinvtor_proj} upon the substitution:
\begin{equation*}
\begin{gathered}
a =x_2 y_1- x_1 y_2,\quad b=- x_2 y_3,\quad c=-x_1y_3,\\
q=x_1y_1,\quad p=-  x_1 y_3+ x_2y_2,\quad r= x_2y_1.
\end{gathered}
\end{equation*}

These equations define a Lie algebra provided \( d^2=0 \), corresponding to the Jacobi identity. Computing \( d^2e^i \), using \eqref{eq:lieinvtor}, we find that this condition can be rephrased in terms of the set of equations
\begin{equation*}
bq-cr=0=ab-br-c^2+cp=ar-cq+pq-r^2.
\end{equation*}
Having excluded the Abelian case, these constraints are equivalent to the condition that
\begin{equation*}
1=\rank \begin{pmatrix} a-r & c & q \\ c-p & b & r \end{pmatrix}=\rank Q_{\lambda,\mu},
\end{equation*}
giving the asserted correspondence with points of \( \Mmod \).

By construction, the subset cut out by \( d^2=0 \) is preserved by the action of \( \SL(2,\bR) \), since the natural map \( \Hom(T^*,\Lambda^2T^*)\to \Hom(\Lambda^2T^*,\Lambda^3T^*) \) is \( \GL(6,\bR) \)-equivariant. Hence, the Jacobi identity defines a \( \GL(6,\bR) \)-invariant subvariety of \( \Hom(T^*,\Lambda^2T^*) \). Its intersection with the \( \SL(2,\bR) \)-module \( B^1\oplus B^3 \) is necessarily preserved by \( \SL(2,\bR) \). By Lemma~\ref{lemma:b1tensorb3}, this action of \( \SL(2,\bR) \) is the standard action on \( \pR^1\times\pR^2 \).

By construction, the projection on \( B^3 \) that gives the intrinsic torsion is given by polynomial multiplication. This is a surjective map because every third degree polynomial has a linear factor.
\end{proof}

Motivated by Proposition \ref{prop:inv_tor_flat_conn}, we  shall refer to \( \Mmod \) as the \emph{invariant torsion variety}. It is the properties of this variety that allows us to complete the proof of Corollary~\ref{cor:locallyhomogeneous}.

\begin{proof}[Proof of Corollary~\ref{cor:locallyhomogeneous}]
Surjectivity of the polynomial multiplication map tells us that any value of invariant intrinsic torsion can be realized in terms of a left-invariant \( \SO(3) \)-structure on a Lie group; this gives the last part of the statement. In addition, by Theorem~\ref{thm:citclassifies}, this implies that any invariant intrinsic torsion \( \SO(3) \)-structure must be locally isometric to a left-invariant structure on a Lie group.
\end{proof}

By the isomorphism \( \Mmod\cong\pR^1\times\pR^2 \) of Proposition \ref{prop:inv_tor_flat_conn}, it follows that, up to the action of \( \GL(2,\bR) \), there are five cases of invariant intrinsic torsion:

\begin{lemma}
\label{lem:5models}
There are precisely five \( \GL(2,\bR) \)-orbits in \( \pR^1\times\pR^2 \), determined by the elements:
\begin{equation*}
u_1\cdot (u_1^2-u_2^2), \quad  u_1\cdot u_2^2, \quad  u_1\cdot u_1^2, \quad  u_1\cdot u_1u_2, \quad  u_1\cdot (u_1^2+u_2^2).
\end{equation*}
\end{lemma}
\begin{proof}
Third degree polynomials are classified according to the number and multiplicity of their real roots and the formal products are determined by choosing a linear factor. Using the action of \( \GL(2,\bR) \), the statement follows.
\end{proof}

In relation to Lemma \ref{lem:5models}, we have five naturally defined subvarieties of \( \Mmod \). Indeed, we can think of the invariant torsion variety in terms of pairs of polynomials, 
\begin{equation}
\label{eq:cls_pol}
\bx =x_1u_1+x_2u_2,\quad \by =y_1u_1^2+y_2u_1u_2 + y_3u_2^2,
\end{equation}
and let 
\begin{equation}
\label{eq:dis_res}
\Delta(\by )=y_2^2-4 y_3 y_1, \quad R(\bx ,\by )= x_2^2 y_1+ y_3 x_1^2- y_2 x_2 x_1
\end{equation}
denote the discriminant of \( \by  \) and the resultant of \(\bx ,\by  \), respectively. Then we can we define \( \Mmod_+\subset\Mmod \) as the open subvariety corresponding to the case when \( \by  \) has two different real roots and \( \gcd(\bx ,\by )=1 \), meaning it is defined by \( \Delta>0 \) and \( R\neq0 \). Similarly, \( \Mmod_-\subset\Mmod \) corresponds to the situation when \( \by  \) has no real roots and \( \gcd (\bx ,\by )=1 \), so that we have \( \Delta<0, R\neq0 \) as defining conditions. A third subvariety, \( \Smod_1\cong\pR^1\times\pR^1 \), occurs when \( \by  \) has a double root, meaning \( \Delta=0 \). Likewise, we can consider the subvariety \( \Smod_2\cong\pR^1\times\pR^1 \) corresponding to the case when \( \bx  \) and \( \by  \) have a common root, i.e., \(R=0\). The intersection  \( \mathcal C=\Smod_1\cap \Smod_2\cong\pR^1 \) corresponds to \( \by  \) being a multiple of \( \bx ^2 \).

A natural question is what are the isomorphism classes of Lie algebras constituting these five varieties. The list of Lie algebras that appear can be anticipated from \cite[Section 8]{Agricola-al:NaturallyReductive}, where the authors classify naturally reductive homogeneous structures in dimension  six. In order to understand this, note that \( 2\bR\subset T^*\otimes \so(3) \) is transverse to \( \Lambda^3T^* \). As a consequence, there is a unique equivariant map
\begin{equation*} 
p\colon 4\bR\subset T^*\otimes\so(3)^\perp\to T^*\otimes \so(3), \quad p(\xi)+\xi\in\Lambda^3T^*. 
\end{equation*}
This means we can modify the canonical connection, using \( p(\tau_\lambda) \), so as to obtain invariant skew-symmetric torsion and, in particular, a naturally reductive homogeneous structure. The adjusted Ambrose-Singer connection has torsion given by the \( 3 \)-form
\begin{multline*}
\tfrac12\lambda_1(e^{235}+e^{145}+e^{136})+\tfrac12(2\lambda_1+\lambda_3) e^{246}\\
-\tfrac12(\lambda_2+2\lambda_4) e^{135}-\tfrac12\lambda_4(e^{146}+e^{236}+e^{245}),
\end{multline*}
and this is an example of an \emph{\( \mathit{SU(3)} \)-instanton} in the sense of \cite{Carrion:instanton}; the \( 2 \)-form part of its curvature takes values in \( \mathfrak{su}(3)\subset\Lambda^2T^*_mM \). In fact, the curvature takes values in \( S^2(\so(3))\subset\Lambda^2T^*_mM\otimes\so(3) \). 

We find that the correspondence between our five varieties and the five different Lie algebras of \cite[Section 8]{Agricola-al:NaturallyReductive} is as follows:
\begin{theorem}
\label{thm:inv_tor_clsf}
The complement of \( \mathcal S_1\cup  \mathcal S_2 \), in the invariant torsion variety, contains exactly two components \( \Mmod_\pm \). The isomorphism classes of points of \( \Mmod \) are then as follows:
\vspace{.5mm}
\begin{compactenum}
 \item points of \( \Mmod_+ \) correspond to the semi-simple Lie algebra \( \so(3)\oplus\so(3) \);
 \item points of\( \Mmod_- \) correspond to the semi-simple Lie algebra \( \so(3,\bC) \);
 \item points of \( \Smod_1\setminus \mathcal C \) correspond to the semi-direct product \( \so(3)\ltimes\bR^3 \);
\item points of \( \Smod_2\setminus \mathcal C \) correspond to the direct sum \( \so(3)\oplus\bR^3 \); 
\item points of \( \mathcal C = \mathcal S_1\cap \mathcal S_2 \) correspond to the nilpotent Lie algebra\\ \( (0,0,0,12,13,23) \).
\end{compactenum}
\end{theorem}

\begin{proof}
By Lemma~\ref{lem:5models}, each of the listed subvarieties is a \( \GL(2,\bR) \)-orbit and therefore connected.

Following \cite{Agricola-al:NaturallyReductive}, one uses the Killing form to distinguish the various Lie algebras and with respect to the basis \( e^1 \), \( e^3 \), \( e^5 \), \( e^2 \), \( e^4 \), \( e^6\), it can be expressed in terms of a block form matrix \( F \); this \( 6\times6 \) matrix consists of four blocks each of which is proportional to the identity matrix, \( \diag(1,1,1) \) and, in addition, the off-diagonal blocks are identical. It follows that \( F \) has at most two distinct eigenvalues, and therefore its rank equals a multiple of three.

The rank of \( F \) is maximal precisely when its determinant does not vanish. Referring to \eqref{eq:dis_res}, we find
\begin{equation*} 
\det F=(4\Delta R^2)^3.
\end{equation*} 
From this expression, it is clear that maximal rank corresponds to the open subvarieties \( \Mmod_\pm \) of \( \Mmod \). In this case \( \g \) is semi-simple and so, by the classification of semi-simple Lie algebras, must be a real form of \( \lsl(2,\bC)\oplus \lsl(2,\bC) \). Hence \( \g \) is either
\( \so(3,\bC)\) or the direct sum of two three-dimensional Lie algebras, each isomorphic to either \( \so(3) \) or \( \lsl(2,\bR) \). Since the signature of the Killing form is either \( (6,0) \) or \( (3,3)\), depending on whether \( \Delta \) is positive or negative, we see that the two situations correspond to \( \so(3)\oplus\so(3) \) or \( \so(3,\bC) \), respectively.

The intersection \( \Smod_1 \cap \Smod_2 \) is the orbit of \( u_1\cdot u_1^2 \). It follows from \eqref{eq:lieinvtor_proj} that each element of this curve in \( \Mmod \) corresponds to the nilpotent Lie algebra \( \g=(0,0,0,12,13,23) \).

The complement of \( \mathcal{C} \) in \( \Smod_1 \) is the orbit of \( u_1\cdot u_2^2 \). This consists of perfect Lie algebras whose radical is a three-dimensional ideal \( \mathfrak{r}=[\g,\g]^\perp=\g^\perp \) spanned by \( e_2, e_4, e_6 \). Computations show that \( \mathfrak r \) is Abelian. The semi-simple quotient \( \g\slash \mathfrak r  \) is a three-dimensional semi-simple Lie algebra whose Lie bracket is \( \SO(3) \)-invariant, hence necessarily \( \so(3) \). Therefore any point of \( \Smod_1\setminus \mathcal C \) corresponds to a Lie algebra which is isomorphic to the semi-direct product \( \so(3)\ltimes\bR^3 \). In this situation, \( \so(3) \) acts non-trivially on \( \bR^3 \), since \( [\g,\g]=\g \).

Finally, note that the complement of \( \mathcal C \) in \( \Smod_2 \) is the orbit of \( u_1\cdot u_1u_2 \).
In this case, points of the orbit have first Lie algebra Betti number equal to three, and
\([\g,\g]\) is easily seen to be isomorphic to \( \so(3) \). In addition, the radical of these Lie algebras, 
\([\g,\g]^\perp \), is an ideal that intersects the derived algebra trivially and so coincides with the center. In conclusion, we have an orbit consisting of Lie algebras \( \g=\so(3)\oplus\bR^3 \).
\end{proof}

The upshot of Lemma \ref{lem:5models} and Theorem \ref{thm:inv_tor_clsf} is the list of model geometries mentioned in the introduction of the paper:

\begin{corollary}
Up to the action of \( \GL(2,\bR) \) there are exactly five non-Abelian model geometries for invariant intrinsic torsion. In terms of pairs of polynomials \( (\bx ,\by ) \), as in \eqref{eq:cls_pol}, these are characterized by Table \ref{table:poly}.
\qed
\end{corollary}

\subsection{Curvature properties}

By Theorem \ref{thm:SO3formsdetermineIT}, one would expect curvature computations to be particularly simple for invariant intrinsic torsion structures. 
We now illustrate how this works by characterizing the locally conformally flat and Einstein metrics. 

It is well known that \( S^3\times S^3 \) admits two left-invariant Einstein metrics, namely the bi-invariant metric and the nearly-K\"ahler metric. The basis of one-forms \( h^1,\dotsc, h^6\) satisfying
\( dh^1=-h^{35} \), \( dh^2=-h^{46} \), and so forth, is orthonormal for the bi-invariant metric. In terms of this basis, an orthonormal coframe for the nearly-K\"ahler metric is given by
\begin{equation*} 
h^1+h^2,\, h^3+h^4,\, h^5+h^6,\, \sqrt3(h^1-h^2),\, \sqrt3(h^3-h^4),\, \sqrt3(h^5-h^6). 
\end{equation*}

There is a, seemingly, different Einstein metric which corresponds to the orthonormal coframe
\begin{equation*}
h^1-h^2,\, h^2,\, h^3-h^4,\, h^4,\,  h^5-h^6,\, h^6.
\end{equation*}
This metric can be distinguished from the first two by considering the ratio between the scalar curvature and the trace of the Killing form; this quantity is an invariant for Lie group isomorphisms that are also isometries and is different for the three metrics.  In terms of the orthonormal frame for the third metric, the structure constants appear as 
\begin{equation}
 \label{eqn:unknowneinsteinmetric}
\begin{gathered}
de^1 = - e^{45}- e^{36}- e^{35},\, de^3 = -e^{52}-e^{61}-e^{51},\, de^5 = - e^{13}- e^{23}- e^{14},\\
de^2 = - e^{46},\, de^4 = -e^{62},\, de^6 = - e^{24}.
 \end{gathered}
\end{equation}

\begin{remark}
The arguments of \cite{Nikonorov:CompacthomogeneousEinstein} imply that the first and third metrics are isometric up to homothety. The isometry, however, is not an isomorphism of Lie groups. 
\end{remark}

Verifiably each of these Einstein metrics corresponds to an \( \SO(3) \)-structure with invariant intrinsic torsion. The converse statement also holds by the second part of the following:
\begin{proposition}
\label{prop:inv_tor_Einstein}
Let \( P \) be an \( \SO(3) \)-structure with invariant intrinsic torsion on \( M \). Then: 
\vspace{.5mm}
\begin{compactenum}
\item the underlying metric is locally conformally flat if and only \( P \) is locally equivalent to \( T^*\bR^3 \);
\item if the underlying metric is non-flat and Einstein, then the scalar curvature is positive and \( (M,P) \) is locally equivalent to a structure on \( \SO(3)\times\SO(3) \). The Einstein metrics correspond to three \( \LU(1) \)-orbits in \( \pR^3 \), and each orbit corresponds to either the bi-invariant metric induced by the Killing form, the nearly-K\"ahler metric or the metric of \eqref{eqn:unknowneinsteinmetric}.  
\end{compactenum}
\end{proposition}

\begin{proof}
Whilst the group \( \SL(2,\bR) \) does not preserve metric properties, it contains a one-dimensional torus which certainly does and under which we have
\begin{equation*}
\begin{gathered}
B^3=\Span{ u_1^3-3u_1u_2^2, u_2^3-3u_1^2u_2}\oplus \Span{u_1^3+u_1u_2^2,u_1^2u_2+u_2^3}=C^3\oplus C^1,\\
B^2=\Span{u_1^2-u_2^2, u_1u_2}\oplus\Span{u_1^2+u_2^2}=C^2\oplus\bR.
\end{gathered}
\end{equation*}
Here \( C^k \) denotes the real irreducible representation of \( \LU(1) \) of weight \( k \). 

An element \( t=((t_1,t_2),(t_3,t_4)) \) of \( C^3\oplus C^1 \) corresponds to the intrinsic torsion 
\( \lambda=(t_1+t_3, -3t_2+t_4,-3t_1+t_3,t_2+t_4) \), and it is convenient to express the curvature in terms of the \( t_i \) rather than the \( \lambda_i \).

A Riemannian six-manifold has curvature that takes values in the module
\begin{equation*} 
\mathcal R=\mathcal W \oplus S^2(\bR^6),
\end{equation*}
which corresponds to the decomposition into the Weyl and Ricci tensors, respectively. 

By Proposition~\ref{Prop:curv_int_tor} and the argument used  in the proof of Theorem~\ref{thm:citclassifies}, the curvature of an \( \SO(3) \)-structure with invariant intrinsic torsion \( \tau \) has the form \( R^{\so(6)}(\tau\otimes \tau, \partial(\tau)\hook\tau) \). More precisely, we have
\begin{equation*}
\mathcal{W}^{\SO(3)}=2\bR\oplus C^4, \quad (S^2T^*)^{\SO(3)}=\bR\oplus C^2.
\end{equation*}
Indeed, computations show that 
\begin{equation*}
(S^2(\Lambda^2(A^2C^1)))^{\SO(3)}=2\bR\oplus S^2(\bR\oplus C^2)=4\bR\oplus C^2\oplus C^4,
\end{equation*}
and \( (\Lambda^4A^2C^1)^{\SO(3)}=\bR \).

By its very construction, the Weyl tensor is a \( \LU(1) \)-equivariant map, 
\begin{equation*} 
S^2(C^3\oplus C^1) \to 2\bR\oplus C^4,
\end{equation*}
which, in particular, means that its two scalar components come from a linear map
\begin{equation*} 
2\bR\subset  S^2(C^3)\oplus S^2(C^1)\to 2\bR.
\end{equation*}
The latter map is injective as can verified by choosing a basis. Consequently any \( P \) which is not (locally) flat will have Weyl tensor with non-trivial component in \( 2\bR\subset\mathcal W^{\SO(3)} \). In particular, \( P \) is locally conformally flat if and only if it is flat and therefore, by Proposition \ref{prop:locflat}, locally equivalent to \( T^*\bR^3 \).  

Considering next the traceless Ricci tensor, a computation shows this takes the form
\begin{multline*}
\Ric_0=-{2( t_4 t_1+2  t_3 t_4+ t_2 t_3)}(e_1\odot e_2+e_3\odot e_4+e_5\odot e_6)\\
+2  {(t_4^{2}-t_3^{2}+ t_3 t_1- t_2 t_4)}((e_1)^2+(e_3)^2+(e_5)^2-(e_2)^2-(e_4)^2-(e_6)^2);
\end{multline*}
the scalar curvature is given by \( s = 5t_1^2+5t_2^2-t_3^2-t_4^2 \).

In terms of the parameterization of Proposition~\ref{prop:inv_tor_flat_conn}, we see that the subset defined by \( \Ric_0=0 \) has the form
\begin{equation*}
\begin{gathered}
- x_1 x_2 y_1^{2}- x_1 y_3^{2} x_2- y_2 x_1^{2} y_1- y_2 y_3 x_2^{2}-6  x_1 y_3 x_2 y_1=0,\\
3  y_3^{2} x_2^{2}-4  x_1^{2} y_3 y_1+4  y_3 x_2^{2} y_1- y_2^{2} x_2^{2}+2  y_2 x_1 y_3 x_2-3  x_1^{2} y_1^{2}=0,\\
-2  y_2 x_1 x_2 y_1- x_1^{2} y_3^{2}+ x_2^{2} y_1^{2}+ y_2^{2} x_1^{2}=0.
\end{gathered}
\end{equation*}
We can eliminate \( [x_1:x_2] \) from these equations by taking the resultant of the two polynomials. Then we get the equation
\begin{equation*}
(y_1 + y_3)^2 (  (y_1+y_3)^2 + 4 y_1 y_3 - y_2^2 )^2 (3 (y_1+y_3)^2 +4 y_1 y_3- y_2^2)=0
\end{equation*}
which defines a \( \LU(1) \)-invariant curve in \( \pR^2 \); this curve has three disjoint components, each isomorphic to \( \pR^1 \).

One verifies that the preimage in \( \Mmod \)  of each component is also connected, and \( \LU(1) \) acts transitively on it since it acts transitively on both the base and fibre. The projection has degree two and one, according to the multiplicity of the defining equation as a factor in the resultant. Thus, up to \( \LU(1) \)-symmetry, we find the possibilities
\begin{equation*}
u_1\cdot (u_1^2-u_2^2), \quad u_1\cdot (u_1u_2+u_2^2), \quad u_1\cdot (u_1^2-3u_2^2).
\end{equation*}

The first of these points yields the Killing metric, the second point gives \eqref{eqn:unknowneinsteinmetric} and, in terms of \eqref{eq:lieinvtor}, the third solution reads 
\begin{equation*} 
[a:p:b:q:c:r]=[0: 3: 0 : 1 : 3:0];
\end{equation*}
this corresponds to the nearly-K\"ahler metric.
\end{proof}

\begin{remark}
From the proof of Proposition \ref{prop:inv_tor_Einstein}, we see that if \( P \) is not (locally) flat, then the Weyl curvature must have non-zero component in the trivial \( \LU(1) \)-modules \( 2\bR\subset\mathcal W^{\SO(3)} \).
\end{remark}

\begin{remark}
The nearly-K\"ahler metric on \( S^3\times S^3 \) can be viewed as an \( \SO(3) \)-invariant Einstein metric on \( \SO(4) \). A generalization to higher dimensions has been  considered in \cite{Pope:HomogeneousEinsteinmetrics}. The fact that Pope's metric on \( \SO(4) \) is the nearly-K\"ahler metric  can be verified by setting \( p=3, q=1 \) in \cite[Equation (16)]{Pope:HomogeneousEinsteinmetrics}.
\end{remark}

\subsection{Half-flat structures and contraction limits}

According to \eqref{eq:SU3_intr_tor} the \( \SU(3) \)-structure associated with an invariant intrinsic torsion \( \SO(3) \)-structu\bdash{}re is half-flat if and only if \( \lambda_2=\lambda_4 \). In this case, the \( \SU(3) \)-intrinsic torsion takes values in \( 3\bR\subset W_1^-\oplus W_3 \) which is not an \( \SL(2,\bR) \)-module; the largest subgroup of \( \GL(2,\bR) \) that fixes this subspace is \( \Sigma_3\times\bR^*\) with \( \Sigma_3 \) acting as in \eqref{eq:sigma3ingl2}. 

In terms of the invariant torsion variety, we have the following geometric description of the half-flat structures:

\begin{proposition}
\label{prop:halfflat}
Regarded as a subset of \( \pR^1\times\pR^2 \), the space of half-flat structures can be characterized as the blow-up of \( \pR^2 \) at \( [1:0:1] \). In particular, it is connected. The points projecting onto \( [1:0:1] \) correspond to Hermitian structures on \( \so(3,\bC) \). 
\end{proposition}

\begin{proof}
Considering the formal products of (non-zero) polynomials \( \bx,\by \) as in \eqref{eq:cls_pol}, the half-flat structures constitute the set
\begin{equation*}
\{ \bx\cdot\by\colon\, x_1y_2+x_2(y_1-y_3)=0 \}. 
\end{equation*}
In terms of the corresponding set of \( \pR^1\times \pR^2 \), the projection map onto the second factor
is an isomorphism away from \( [1:0:1]\in\pR^2 \). By choosing an affine chart centered at this point, the first part of the proposition is readily verified.

For the last assertion, we note that the product \( \bx\cdot(y_1u_1^2+y_2u_2^2) \) has \( \Delta=-4 \), \( R=x_1^2+x_2^2 \) and \( \lambda_1=x_1=\lambda_3 \). It follows that points in the preimage of \( [1:0:1] \) correspond to \( \so(3,\bC) \) and that the associated \( \SU(3) \)-structure has \( W_1^-=0 \), hence is Hermitian. 
\end{proof}

We shall now describe a way of deforming half-flat structures on \( \Mmod \). It is a type of Abelianization procedure which is sometimes referred to as \emph{Lie algebra contractions} (see, for instance, \cite{Gibbons-al:domainwalls}). To this end, we think of Lie algebras as points of the variety
\begin{equation}
\label{eq:Liealgvar}
\Dmod =\left\{d\colon (\bR^6)^*\to \Lambda^2(\bR^6)^* \mid \hat d \circ d=0\right\}, 
\end{equation}
consisting of Chevalley-Eilenberg operators \( d\colon\,\Lambda(\bR)^*\to\Lambda(\bR^6)^* \). Regarding elements of \( \Dmod \) as linear isomorphisms \( \bR^6\to\g \) to a Lie algebra, we note that \( \GL(6,\bR) \) acts naturally on \( \Dmod \) in two ways; we can use either left or right composition on the isomorphism. From our view point, the right action is distinguished by being independent of the underlying Lie algebra and compatible with the natural action on \( \Hom((\bR^6)^*,\Lambda^2(\bR^6)^*) \). In addition, the intrinsic torsion map is equivariant with respect to this action.
In these terms, a contraction limit is a procedure by which one obtains an element of \( \Dmod \) as the limit point of a curve \(t\mapsto d\cdot g_t\), \( g_t\in\GL(6,\bR) \).

A tractable subclass of contraction limits are those curves that arise as one-parameter subgroups. We restrict our attention further, namely to subgroups of \( \GL(2,\bR) \); this group preserves the invariant intrinsic torsion condition. It is then harmless to consider only subgroups of \( \SL(2,\bR) \), since the rescalings commute with \( \SL(2,\bR) \) so as to give rise to homothetic structures on the same Lie algebra.

The action of a one-parameter subgroup is entirely determined by the fundamental vector field associated to a generator
\begin{equation*}
A_{a,b,c}=\begin{pmatrix} a & b \\ c & -a \end{pmatrix} \in \lsl(2,\bR). 
\end{equation*}
Explicitly, the fundamental vector field on \( B^3 \) associated to \( A_{a,b,c} \) takes the form:
\begin{multline*}
X_{a,b,c}= (3a\lambda_1 + b\lambda_2)\D{}{\lambda_1}+ (3c\lambda_1+a\lambda_2 + 2b\lambda_3)\D{}{\lambda_2} \\
 + (2c\lambda_2-a\lambda_3+3b\lambda_4)\D{}{\lambda_3} + (c\lambda_3 -3a\lambda_4)\D{}{\lambda_4}.
\end{multline*}
Obviously, \( X_{a,b,c} \) will generally fail to preserve the half-flat condition. In fact, up to the action of \( \Sigma_3\times\bR^* \), there are precisely three elements of \( \lsl(2,\bR) \) that generate a flow of half-flat structures:  

\begin{proposition}
\label{prop:contr_lim}
Up to the action of \( \Sigma_3\times\bR^* \) there are three one-parameter subgroups in \( \GL(2,\bR) \) with a non-trivial orbit of half-flat structures. These are generated by the vector fields
\( X_{1,0,0} \), \( X_{0,1,3} \) and \( X_{0,1,-1} \).

The union of half-flat orbits is in each case a two-plane in \( B^3 \) which is given by 
\begin{equation*}
\begin{gathered}
\Pi_{1,0,0}=\left\{\lambda_2=\lambda_4=0\right\},\, \Pi_{0,1,3}=\left\{9\lambda_1=\lambda_3, \lambda_2=\lambda_4\right\},\\
 \Pi_{0,1,-1}=\left\{\lambda_1=\lambda_3, \lambda_2=\lambda_4\right\},
\end{gathered}
\end{equation*}
respectively.

The plane \( \Pi_{1,0,0} \) contains a representative of each Lie algebra appearing in Theorem~\ref{thm:inv_tor_clsf}. The second plane, \( \Pi_{0,1,3} \), contains only \( \so(3,\bC) \) and the nilpotent algebra. The third, \( \Pi_{0,1,-1} \), runs through \( \SU(3) \) reductions of a fixed Hermitian structure on \( \so(3,\bC) \) that only differ by a rotation of the complex volume form.
\end{proposition}

\begin{proof}
As \( B^3 \) has weights one and three the Cartan subalgebra generated by \( A_{a,b,c} \) must act with multiplicities two and four. It follows that the largest subspace of \( B^3 \) which is both contained in \( \{\lambda_2=\lambda_4\} \) and is invariant under \( X_{a,b,c} \) (if non-trivial) must be two-dimensional. This maximal subspace is contained in the two-plane where \( X_{a,b,c}(\lambda_2-\lambda_4) \) vanishes, leading to the following:
\begin{equation*}
\Pi_{a,b,c}=
\begin{cases}
3c\lambda_1 + 4a\lambda_2+(2b-c)\lambda_3=0\\ \lambda_2=\lambda_4
\end{cases}.
\end{equation*}

The condition on the vector field \( X_{a,b,c} \) to be tangent to this plane is found from the expression
\begin{multline*}
X_{a,b,c}(3c\lambda_1 + 4a\lambda_2+(2b-c)\lambda_3)=2(3b^2-c^2+2bc-12a^2)\lambda_2\\
+8a(-b+c)\lambda_3.
\end{multline*}
This gives us three situations to consider: 
\vspace{.5mm}
\begin{compactenum}
\item\label{contr:itm1} if \( c\neq0 \), we must have
\( 3b^2-c^2+2bc-12a^2=0 \) and \( a(b-c)=0 \) which implies one of the following possibilities: \( a=0 \), \( b=-c \) or \( a=0 \), \( b=c\slash 3 \) or \( b=c \), \( b=\pm\sqrt3 a \);
\item\label{contr:itm2} if \( c=0 \) one possibility is to have \( b=0\neq a\);
\item\label{contr:itm3} alternatively, for \( c=0 \), we can have \( b\neq0 \). In that case, we find that 
\( a=\pm\frac{\sqrt3}2\,b \).
\end{compactenum}
\vspace{1mm}

Considering the first of the three possibilities from \eqref{contr:itm1}, we have generating vector field
\( X_{0,1,-1} \) which means \( \lambda_1=\lambda_3 \). The corresponding plane, \( \Pi_{0,1,-1} \), then consists of half-flat structures on \( \so(3,\bC) \) that have \( W_1^\pm=0 \), and the one-parameter group generated by \( X_{0,1,-1} \) is the standard \( \LU(1) \) action which rotates the complex volume form \( \gamma+i\hat\gamma \).

For the second subcase of \eqref{contr:itm1}, the generating vector field is \( X_{0,1,3} \). On the corresponding plane \( \Pi_{0,1,3} \), we have \( 9\lambda_1=\lambda_3 \), i.e.,
 \begin{equation*}
\begin{pmatrix} x_2 & x_1 & -x_2 \\ 9x_1 & - x_2 & -x_1\end{pmatrix}
 \begin{pmatrix} y_1\\ y_2\\ y_3\end{pmatrix}=0.
\end{equation*}
Explicitly, we can write
\begin{equation*}
(x_1u_1+x_2u_2)\cdot((x_1^2+x_2^2)(u_1^2+9u_2^2)+8x_1x_2u_1u_2).
\end{equation*}
This family has \( \Delta\leqslant0 \) and contains only two points which are not semi-simple. They correspond to
\( [x_1:x_2]=[1:\sqrt3] \) or \( [x_1:x_2]=[1:-\sqrt3] \), and both give the nilpotent Lie algebra \( (0,0,0,12,13,23) \).

Up to the action of \( \Sigma_3\times\bR^* \), the vector field \( X_{1,\pm\sqrt3/2,\pm\sqrt3/2} \) appearing in the final subcase of \eqref{contr:itm1} is equivalent to \( X_{1,0,0} \) of case \eqref{contr:itm2}. Let us therefore consider the generating vector field \( X_{1,0,0} \). This has \( \lambda_2=0 \) and therefore gives rise to the following possibilities 
\begin{equation*}
 u_1\cdot(y_1u_1^2+y_3u_2^2),\quad   (x_1u_1+x_2u_2)\cdot(x_1u_1^2-x_2u_1u_2).
\end{equation*} 
The first of these points has \( \Delta=-4y_1y_3 \) and \( R=y_3 \), so that we obtain all the types of Lie algebras appearing in \( \Mmod \), except \( \SO(3)\times\bR^3 \). 
Similarly, the second point gives \( \Delta=x_2^2 \) and \( R=2x_2^2x_1 \) which leads to all the possible types of Lie algebras apart from \( \SO(3)\ltimes\bR^3 \).

The proof is concluded by noticing that the remaining case \eqref{contr:itm3}, with generating vector field \( X_{\sqrt3/2,\pm1,0} \), is equivalent to \( X_{0,1,3} \).
\end{proof}

\begin{example}
By integrating the vector field  \( X_{1,0,0} \) we obtain the one-parameter of elements of the form 
\( \diag(\lambda,\lambda^{-1}) \) with \( \lambda\in\bR^* \). 

Acting with such an element maps \( u_1\cdot(u_1^2\pm u_2^2) \) to \( u_1\cdot(\lambda^2u_1^2\pm \lambda^{-2}u_2^2) \).
In this way we obtain the following two curves
\begin{equation*}
\mathcal C_{\pm}=\left\{u_1\cdot(\lambda^2u_1^2\pm \lambda^{-2}u_2^2)\mid\,\lambda\in\bR_+\right\}.
\end{equation*}
Both of these correspond to the semi-direct product \( \so(3)\ltimes\bR^3 \) as \( \lambda\to0 \) and the nilpotent Lie algebra as \( \lambda\to+\infty \). Points in \( \lambda\in(0,+\infty) \) correspond to either \( \SO(3,\bC) \) or \( S^3\times S^3 \) depending on whether they are in \( \mathcal C_+ \) or \( \mathcal C_- \), respectively.

On the other hand,  if we let \( \diag(\lambda,\lambda^{-1}) \) act on 
\begin{equation*}
(x_1u_1+x_2u_2)\cdot(x_1u_1^2-x_2u_1u_2), 
\end{equation*}
the resulting product is 
\begin{equation*}
(\lambda x_1u_1+\lambda^{-1} x_2u_2)\cdot(\lambda x_1u_1^2-\lambda^{-1}x_2u_1u_2).
\end{equation*}
In this way, we obtain two curves that connect the nilpotent algebra and \( \so(3)\oplus\bR^3 \) through \( \so(3)\oplus \so(3) \).
\end{example}

\begin{remark}
In the next section we shall discuss another flow of half-flat structures which is also generated by the action of \( \GL(2,\bR) \).
This flow  is unrelated to the above contractions in the sense that it preserves \( \Pi_{1,0,0} \), but not the other two-planes of Proposition \ref{prop:contr_lim}. 
\end{remark}

\section{Uniqueness of the Bryant-Salamon metric}
\label{sec:half-flat-evol}

A six-manifold \( M \) with an \( \SU(3) \)-structure that is half-flat can, at least if real analytic, be embedded in a manifold \( N \) with holonomy contained in \( \G_2 \) \cite{Hitchin:Stableformsand} (see \cite{Bryant:Non-embedding} for counterexamples outside the real analytic setting). The metric on \( N \) is Ricci-flat and can be found by solving a system of evolution equations which is usually referred to as the Hitchin flow. Explicitly, we have 

\begin{equation}
\begin{cases}
\label{eq:Hitchin_flow}
\gamma'=d\sigma,\\ 
(\sigma^2)'=-2d\hat{\gamma}.
\end{cases}
\end{equation} 

As already discussed, the half-flat structures are characterized by the equations \eqref{eq:SU3_halfflat} whose symmetries consist of elements in \( \GL(6,\bR) \) that preserve \( \Span{\sigma^2} \) and \( \Span{\gamma} \). This group contains \( \SU(3) \), dilations \( \bR_+ \) and the group of order two corresponding to complex conjugation.

The algebraic variety of six-dimensional Lie algebras with a fixed \( \SU(3) \)-structure can be identified with the \( \SU(3) \) quotient of the space 
\( \Dmod \) defined in \eqref{eq:Liealgvar}. Focusing on the subvariety of half-flat structures leaves us with a quotient of
\begin{equation*} 
\Nmod  = \left\{ d\in\Dmod \mid d\gamma=0=d\sigma^2\right\}. 
\end{equation*} 

As observed in \cite[Remark 2.10]{Belgun:Ontheboundary} any subgroup of \( \SU(3) \) will preserve \eqref{eq:Hitchin_flow}. In particular, we have

\begin{lemma}
\label{lem:Hit_subgrp}
For any subgroup \( \LH \) of \( \SU(3) \), \( \mathcal{N}^{\LH} \) is invariant under the Hitchin flow. In particular, this flow leaves invariant the subset \( \Nmod^{\SO(3)} \) of the invariant torsion variety characterized by the condition that \( W_1^+=0 \).
\end{lemma} 

\begin{remark}
It is natural to ask whether the Hitchin flow on \( \Nmod^{\SO(3)} \) extends to the whole of \( \Mmod \). The answer is (essentially) affirmative in the sense that away from the Hermitian structures on \( \so(3,\bC) \), it extends to a \( \LU(1) \)-invariant flow. This follows from the fact that, in this case, there is a uniquely associated half-flat structure (modulo an action of \( \bZ_3 \)). 
\end{remark}

\subsection{Integrating the flow}

Fixing a point \( d\in \Nmod^{\SO(3)} \) determines a Lie algebra \( \g \) together with a fixed frame \( u\colon\bR^6\to\g \). If we denote also by \( u \) the induced maps \( \Lambda^k(\bR^6)^*\to \Lambda^k\g^* \), then we have induced forms on \( \g \) given by 
\begin{equation*}
\gamma=u(\gamma_{\bR^6}), \quad \sigma=u(\sigma_{\bR^6}), 
\end{equation*}
where the subscript refers to the ``standard'' forms on \( \bR^6 \).

For each \( g \in \GL(2,\bR) \), \( gd \) obviously represents the same Lie algebra as \( \g \), but with a different frame \( ug \). Consequently, we obtain a map
\begin{equation*} 
\GL(2,\bR)\to \Lambda^3\g^*, \quad g\mapsto (ug)(\gamma_{\bR^6}) = R_g(\gamma), 
\end{equation*}
and, similarly,
\begin{equation*} 
\GL(2,\bR)\to \Lambda^2\g^*, \quad g\mapsto (ug)(\sigma_{\bR^6}) =R_g(\sigma). 
\end{equation*}
By construction, \( R_g(\sigma)  = \det(g^{-1})\sigma\).

Because we are working with differential forms, it will be more natural to use the associated left action given by 
\( L_g\rho = R_{g^{-1}}\rho \). Explicitly, we parameterize \( \GL(2,\bR) \) as
\begin{equation}
\label{eq:GL2param}
g=\begin{pmatrix} x & y \\ z & w \end{pmatrix}\in\GL(2,\bR),
\end{equation}
so that if \( u \) corresponds to the coframe \( (e^1,\dots, e^6) \), then \( ug^{-1} \) corresponds to 
\( xe^1+ye^2 \), \( ze^1+we^2 \), and so forth. In particular, we have
\begin{multline*}
L_g(\gamma) = (x^3-3xz^2) e^{135}+(xy^2-xw^2-2yzw)(e^{146}+e^{236}+e^{245})\\
+(y^3-3 yw^2)e^{246}+(x^2y-2xzw-yz^2)(e^{235}+e^{145}+e^{136}). 
\end{multline*}
Note that the left action of \( \Sigma_3 \) on \( \GL(2,\bR) \) leaves the flow invariant, because \( u(k g)^{-1} = (ug^{-1})k^{-1} \).  Similarly, the right action of \( \operatorname{Aut}(u,\g) \) on \( \GL(2,\bR) \) preserves the flow.

Recall that the intrinsic torsion of  \( (u,\g ) \) is determined by
\begin{equation}
\label{eq:halfflat_intrtor}
d\sigma =  3 \lambda_1e^{135}+ \lambda_2(e^{235}+e^{145}+e^{136})+\lambda_3(e^{146}+e^{236}+e^{245}) 
+3  \lambda_4 e^{246}.
\end{equation}
As the intrinsic torsion is an element of \( B^3\cong\bR_3[u_1,u_2] \), we may represent this by a polynomial
\begin{equation*}
p=\lambda_1u_1^3+\lambda_2u_1^2u_2+\lambda_3u_1u_2^2+\lambda_4u_2^3.
\end{equation*}

\begin{lemma}
Consider \( (u,\g) \) with intrinsic torsion \( p \) and \( g \) as in \eqref{eq:GL2param}. Then
\( (ug^{-1},\g) \) is half-flat if and only if 
\begin{equation*} 
p(y,-x) - y\D{p}{u_1}(w,-z)+x\D{p}{u_2}(w,-z)=0.
\end{equation*}

Moreover, it is Hermitian if and only if it is half-flat and 
\begin{equation}
\label{eq:Hermloc}
 p(w,-z)-w\D{p}{u_1}(y,-x)+z\D{p}{u_2}(y,-x)=0. 
\end{equation}
\end{lemma}

\begin{proof}
The intrinsic torsion of \( ug^{-1} \) is computed by using \eqref {eq:halfflat_intrtor}. We find
\begin{equation}
 \label{eqn:lambdafromg}
\begin{gathered}
 \lambda_1(ug^{-1})= -\lambda_4z^3+\lambda_3z^2w- \lambda_2zw^2+\lambda_1w^3,\\
 \lambda_2(ug^{-1})= 3\lambda_4xz^2-2\lambda_3xzw + \lambda_2xw^2- \lambda_3yz^2 +2 \lambda_2yzw-3\lambda_1yw^2,\\
 \lambda_3(ug^{-1})= -3 \lambda_4x^2z+ \lambda_3 x^2w+2 \lambda_3xyz-2\lambda_2xyw-\lambda_2y^2z+3 \lambda_1y^2w,\\
\lambda_4(ug^{-1})= \lambda_4x^3- \lambda_3 x^2y+ \lambda_2xy^2-\lambda_1y^3.
\end{gathered}
\end{equation}

The half-flat condition, \( \lambda_2(ug^{-1})=\lambda_4(ug^{-1}) \), therefore reads
\begin{multline*}
0= \lambda_1(y^3-3yw^2)-\lambda_2(xy^2-xw^2-2yzw)\\+\lambda_3(x^2y-2xzw-yz^2)-\lambda_4(x^3-3xz^2)\\
 = p(y,-x) - y\D{p}{u_1}(w,-z)+x\D{p}{u_2}(w,-z),
  \end{multline*}
as claimed.

In our setting, Hermitian structures are half-flat and, additionally, satisfy the condition \( \lambda_1(ug^{-1})=\lambda_3(ug^{-1}) \). Similar computations to the above then verify the expression 
\eqref{eq:Hermloc}. 
\end{proof}

We can use the discriminant of a third degree polynomial 
\begin{equation*}
\begin{gathered}
q=q_1u_1^3+q_2u_1^2u_2+q_3u_1u_2^2+q_4u_2^3,\\
\Delta(q)=q_2^2 q_3^2 - 4 q_1 q_3^3 - 4 q_2^3 q_4 + 18 q_1 q_2 q_3 q_4 - 27 q_1^2 q_4^2,
\end{gathered}
\end{equation*}
so as to describe the Hitchin flow as a linear flow on the space of polynomials; the discriminant then determines the velocity.
Explicitly, we define the map \( Q\colon \gl(2,\bR)\to \bR_3[u_1,u_2]\) by
\begin{equation*} 
\begin{pmatrix} x & y \\ z & w \end{pmatrix} \mapsto \tfrac13(xu_1+yu_2)^3 - (xu_1+ yu_2)(zu_1+wu_2)^2,
\end{equation*}
whose image is the space of polynomials with either three distinct roots or one triple root. 

Indeed, if \( xu_1+yu_2 \) and \( zu_1+wu_2 \) are linearly dependent, the image has a triple root. If they are independent, then the image of \( Q \) has three distinct roots. Conversely, any polynomial with three distinct roots can be written as  the product of three linear factors \( f_1f_2f_3 \) such that \( f_1+f_2+f_3=0 \) and the equations
\begin{equation*}
(\tfrac34)^{1/3}f_1=xu_1+yu_2, \quad (48)^{-1/6}(f_2-f_3)=zu_1+wu_2 
\end{equation*}
determine a unique matrix, up to the \( \Sigma_3 \) action that permutes the linear factors. It follows that the restriction of \( Q \) to the set of invertible matrices defines a \( 3 \)-fold covering map
\begin{equation}
\label{eqn:Qcovering}
Q\colon \GL(2,\bR)\to \{q\in\bR_3[u_1,u_2]\mid \Delta(q)>0\}.
\end{equation}
This map is equivariant with respect to the right action of \( \GL(2,\bR) \), and the subgroup \( \Sigma_3 \) generated by \eqref{eq:sigma3ingl2} and acting on the left is its group of automorphisms.

\begin{lemma}
\label{lemma:hitchinpoly}
Let \( (ug^{-1}(t),\g) \) be a one-parameter family of half-flat structures that satisfy Hitchin's evolution equations and let \( p \) be the intrinsic torsion of \( (u,\g) \). If we set
\( q(t)=Q(g(t)) \) then
\begin{equation*}
q'(t)=\det(g) p,
\end{equation*}
where \( (\det g(t))^6=\tfrac34 \Delta(q(t)) \). In particular, \( q(t) \) describes a line interval of the form \( (q+s_- p,q+s_+p) \). If \( s_\pm \) is finite, \( \Delta(q+s_\pm p)=0 \).

In addition, if the line interval does not contain any Hermitian structure, then \( \Delta(q(t)) \) is strictly monotonic.
\end{lemma}

\begin{proof}
A solution, \( u(t)=ug^{-1}(t) \), to the evolution equations \eqref{eq:Hitchin_flow} satisfies
\begin{equation*}
\begin{gathered}
(L_{g(t)}\gamma)' = d (L_{g(t)}\sigma ) = \det(g(t))d\sigma,\\
L_{g(t)}\det(g(t))'\sigma^2=-d(L_{g(t)}(\hat\gamma)).  
\end{gathered}
\end{equation*}
The second equation can be rewritten in the form
\begin{equation*}
(\det g(t))^2\det(g(t))' \sigma^2=(\lambda_1(t)-\lambda_3(t))\sigma^2\slash2,
\end{equation*}
which leads to
\begin{equation*}
\tfrac23((\det g)^3)'=p(w,-z)-w\D{p}{u_1}(y,-x)+z\D{p}{u_2}(y,-x);
 \end{equation*}
by \eqref{eq:Hermloc} the right-hand side vanishes if and only if the structure is Hermitian.

Explicit computations show that the coefficients of \( q(s) \) represent the coefficients of \( L_g\gamma \), with respect to a suitable basis, 
\begin{multline*}
L_g\gamma =  (x^3-3xz^2)e^{135} +  (x^2y-2xzw -yz^2){(e^{235}+e^{145}+e^{136})}\\
+ (xy^2-xw^2-2zyw)(e^{146}+e^{236}+e^{245})+ (y^3-3yw^2)e^{246},
\end{multline*}
and that the discriminant is related to \( \det g \) as stated.

The fact that \( q \) evolves inside the affine line,
\begin{equation*}
\{q+s p \mid  \Delta(q+s p)\geqslant 0\},
\end{equation*}
follows from the equations. In fact, it ranges in a connected component of this space (hence a line interval), by surjectivity of \eqref{eqn:Qcovering}.

Clearly, the second evolution equation shows that the discriminant is either strictly increasing or strictly decreasing away from the Hermitian locus \eqref{eq:Hermloc}.
\end{proof}

\begin{remark}[Classifying according to symmetry]
\label{remark:changing}
Suppose that \( (ug^{-1}(t),\g) \) is a solution to the Hitchin flow and \( (u,\g) \) has intrinsic torsion \( p \). For any \( k\in\GL(2,\bR) \) the same solution can be written as \( (uk^{-1} (gk^{-1})^{-1}(t),\g) \), which means the reference Lie algebra is changed from \( (u,\g) \) to \( (uk^{-1},\g) \) and the integral line in \( B^3 \) from \( q+sp \) to \( k q + sk p \). Provided \( k p=p \), this is also an integral line relative to \( (u,\g) \). In particular, integral lines on \( (u,\g) \) come in families determined by the stabilizer of \( p \) in \( \GL(2,\bR) \). 

Since \( p \) has three distinct linear factors, its stabilizer is the \( \Sigma_3 \) that permutes them; we just identified this with the group of automorphisms of the map \( Q \), or, more explicitly, the \( \Sigma_3 \) of \eqref{eq:sigma3ingl2} acting on \( \GL(2,\bR) \) on the left. The action on \( B^3 \), however, comes from right multiplication on \( \GL(2,\bR) \), and we therefore obtain a \( \Sigma_3 \) which is only conjugated to our ``standard'' one.

One can then categorize integral lines \( \{ q+sp \} \) according to their stabilizer in \( \Sigma_3 \), which depends only on the action on \( q \). In this terminology, the nearly-K\"ahler metric is characterized by \( \Sigma_3 \)-invariance, and the Bryant-Salamon metric is \( \Sigma_2 \)-invariant. 
\end{remark}

\begin{remark}
In  \cite{Hitchin:Stableformsand} Hitchin interpreted the flow \eqref{eq:Hitchin_flow} of initial data \( (\sigma_0,\gamma_0) \) as a Hamiltonian flow on the product of cohomology classes \( [\gamma_0]\times[\frac12\sigma_0^2] \). By restricting this space to \( \SO(3) \)-invariant forms, we obtain the product of two affine lines parameterized by
\begin{equation*} 
\gamma=\gamma_0+a_1d\sigma, \quad \sigma^2/2=(1+a_2)\sigma_0^2\slash2. 
\end{equation*}
When fixing \( -\sigma_0^3\slash2 \) as a reference volume form, the canonical symplectic structure on  \( [\gamma_0]\times[\frac12\sigma_0^2] \) restricts to
\begin{equation*} 
\langle (a_1',a_2'),(b_1',b_2')\rangle = a_2'b_1'-a_1'b_2'. 
\end{equation*}
The Hamiltonian for the Hitchin flow is then given by 
\begin{equation*}
\mathcal{H}= \mathcal{V}(\gamma)-2\mathcal{V}(\sigma^2/2), 
\end{equation*}
where \( \mathcal{V} \) denotes the volume associated to a stable form as in \cite{Hitchin:Stableformsand}. Explicitly these volumes are \( \mathcal{V}(\sigma^2/2)=(1+a_2)^{3/2} \) and
\begin{multline*}
\mathcal{V}(\gamma)^2=4+12  a_1 (\lambda_1-\lambda_3)-12  a_1^2(3  \lambda_1 \lambda_3+2 \lambda_2^2-\lambda_3^2)\\
-4 a_1^3(27  \lambda_1 \lambda_2^2-9  \lambda_1 \lambda_3^2-3  \lambda_2^2 \lambda_3+\lambda_3^3)\\
-3a_1^4(27  \lambda_1^{2} \lambda_2^2-18  \lambda_1 \lambda_2^2 \lambda_3+4  \lambda_1 \lambda_3^3+4 \lambda_2^4- \lambda_2^2 \lambda_3^2),
\end{multline*}
where we have imposed the half-flat condition, meaning \( \lambda_2=\lambda_4 \). 

The associated Hamiltonian equations now read
\begin{equation*} 
a_1'=\D{\mathcal{H}}{a_2}=\sqrt{1+a_2}, \quad a_2'=-\D{\mathcal{H}}{a_1}, 
\end{equation*}
and in the notation of Lemma~\ref{lemma:hitchinpoly} we have \( \det g = \sqrt{1+a_2} \).  This means the ``velocity'' determined by the discriminant of \( q\in \bR_3[u_1,u_2] \) is, in fact, the time derivative of the position variable in the phase space where \( \gamma \) represents position and \( \sigma^2\slash2 \) conjugate momentum.
\end{remark}

\begin{remark}
Changing the variable to \( s \) such  that \( ds\slash dt=\det g(s) \), we can write 
\begin{equation*} 
q(s)=q(0)+sp, \quad  t=\int_0^s \frac{ds}{\det g(s)}. 
\end{equation*}
We may then assume that \( g(0)=1 \) so that \( q(0)=\tfrac13u_1^3 - u_1u_2^2 \). In some cases, however, it is better to take \( g(0) \) arbitrary and instead fix a representative for \( p \) in each \( \GL(2,\bR) \)-orbit.
\end{remark}

Let \( (ug^{-1}(t),\g) \) be a maximal solution to the Hitchin flow with \( g\colon (a,b)\to\GL(2,\bR) \). We say that \( g(t) \) is defined on \( [a,b) \) if the limit \( \lim_{t\downarrow
 a} g(t) \) exists,  similarly for \( (a,b] \).

\begin{lemma}
\label{lemma:endpoint}
Let \( q(t) \) be a maximal solution to the Hitchin flow on \( \bR_3[u_1,u_2] \) defined on an interval \( [0,t_+) \) or \( (t_-,0] \). Then either \( q(0)=0 \), or \( q(0)=f^3 \) with \( f \) a linear divisor of \( p \). More precisely, up to \( \GL(2,\bR) \)-symmetry, the cases are:
\vspace{.5mm}
\begin{compactenum}
\item\label{itm1:endpoint}  \( p=u_1(u_1^2+u_2^2) \) and \( q=\lambda u_1^3 \) with \( \lambda\neq0 \);
\item\label{itm2:endpoint}  \( p=u_1u_2^2 \) and  \( q=\lambda u_1^3 \) with \( \lambda\neq0 \);
\item\label{itm3:endpoint}  \( p=u_1(u_1^2-u_2^2) \) and \( q \) has the form 
 \( \lambda u_1^3 \), \( \lambda (u_1+u_2)^3 \) or \( \lambda (u_1-u_2)^3 \) for some \( \lambda\in\bR \).
\end{compactenum}
\end{lemma}

\begin{proof}
Let us write \( q=q(0)=Q(\begin{smallmatrix}x&y\\z&w\end{smallmatrix}) \). By maximality, \( q \) has discriminant zero. If \( q \) is zero the first claim holds trivially. Otherwise, we have 
\begin{equation*} 
zu_1+wu_2=h(xu_1+yu_2), \quad h\in\bR, 
\end{equation*}
and the half-flat condition then gives
\begin{equation*} 
0=p(y,-x)-h^2\left(y\D{p}{u_1}(y,-x) - x\D{p}{u_2}(y,-x)\right)=p(y,-x)-3h^2p(y,-x). 
\end{equation*}
So if \( xu_1+yu_2 \) does not divide \( p \), we must have \( h=\pm\tfrac1{\sqrt3} \) which implies \( q=0 \). This proves the first part of the statement.

There are four  orbits for the action of \( \GL(2,\bR) \) on \( \bR_3[u_1,u_2] \), hence four cases to consider. However, the situations when
\begin{equation*}
\begin{gathered}
p=u_1^3, q=\lambda u_1^3,\quad  p=u_1u_2^2, q=\lambda u_2^3 \quad \textrm{and}\quad p=u_1(u_1^2+u_2^2), q=0
\end{gathered}
\end{equation*}
can be ruled out because, in each case, the affine line \( \{q+sp\} \) does not contain polynomials with positive discriminant.
In this way we arrive at the possibilities \eqref{itm1:endpoint}--\eqref{itm3:endpoint}. 
\end{proof}

\begin{example}
In general, the affine line \( \{ q+s p \} \) may contain more than one polynomial with discriminant zero, even if only one point in the line can be the boundary point for a solution to the Hitchin flow.  
In fact, this flow always extends continuously to the boundary inside \( \bR_3[u_1,u_2] \), because it occurs inside an affine line. This, however, does not imply it extends on \( \GL(2,\bR) \).

Consider, for instance, the polynomial 
\( q(t)=u_1^3(t+6)-tu_1 u_2^2 \). For \( t>0 \), we can write this as 
\begin{equation*} 
u_1(u_1^2(t+6)-tu_2^2) = u_1(\sqrt{t+6}\,u_1+\sqrt t \,u_2)(\sqrt{t+6}\,u_1-\sqrt t\, u_2), 
\end{equation*}
leading to 
\begin{equation*}
g=\diag((3(t+6))^{1/3},-(3(t+6))^{-1/6}), \quad \det g = 3(t+6)^{1/6}\sqrt t.
\end{equation*}

On the other hand, when \( t<-6 \), we have 
\begin{equation*} 
u_1(u_1^2(t+6)-tu_2^2) = -u_1(\sqrt{-t-6}\,u_1+\sqrt {-t} \,u_2)(\sqrt{-t-6}\,u_1-\sqrt {-t}\, u_2) 
\end{equation*}
so that 
\begin{equation*}
g=\diag( (3(t+6))^{1/3},(3(-t-6))^{-1/6}), \quad \det g = 3(-t-6)^{1/6}\sqrt{-t}; 
\end{equation*}
this goes to infinity at the boundary point. 

In conclusion, we have a one-parameter family of polynomials defined for \( t\in(-\infty,-6]\cup [0,+\infty) \), but in terms of \( \GL(2,\bR) \) the family is only well-defined on \( (-\infty,-6)\cup [0,+\infty) \).
\end{example}

\subsection{The classification}
\label{subsection:cohom1}
Let \( \G \) be a Lie group with a half-flat \( \SO(3) \)-structure that has invariant intrinsic torsion. A maximal solution of the Hitchin flow determines a cohomogeneity one \( \G_2 \)-metric on a product \( \G\times(t_-,t_+) \). We want to determine necessary conditions for \( \G\times (t_-,t_+) \) to embed in a complete cohomogeneity one manifold. 

The possibility that \( (t_-,t_+) \) is the whole real line can be ruled out by the following lemma; it allows us to deduce there are no non-flat complete solutions to the Hitchin flow which are defined on \( \G\times\bR \).

\begin{lemma}
\label{lem:alt_CG}
Let \( p\in\bR_3[u_1,u_2] \) be a non-zero polynomial satisfying the half-flat condition, 
meaning \( p=\lambda_1 u_1^3 + \lambda_2(u_1^2u_2+u_2^3)+\lambda_3 u_1u_2^2 \). Then the affine line
\begin{equation}
\label{eqn:affineline}
\left\{\tfrac13u_1^3-u_1u_2^2+sp\mid s\in\bR\right\}
\end{equation}
contains at least one polynomial with a root of multiplicity greater than one, equivalently with vanishing discriminant.
\end{lemma}

\begin{proof}
Suppose for contradiction that \( \Delta>0 \) on the whole affine line \eqref{eqn:affineline}. Then we obtain a solution to the Hitchin flow for initial data which has intrinsic torsion given by \( p \). 

If the line \eqref{eqn:affineline} does not contain any Hermitian structures, then \( \Delta \) is strictly monotonic by Lemma~\ref{lemma:hitchinpoly}. This implies that \( \Delta(\tfrac13u_1^3-u_1u_2^2+sp) \) is a monotonic polynomial in \( s \) with no roots, which is absurd.

If the line contains a Hermitian structure, then this can be chosen as the initial point of the flow. In terms of polynomials this amounts to acting by an element of \( \GL(2,\bR) \) which preserves the sign of \( \Delta \). Then \( p \) has the form 
\begin{equation*}
p=\lambda_1(u_1^3+u_1u_2^2)+\lambda_2(u_1^2u_2+u_2^3).
\end{equation*}
A computation now shows that
\begin{equation*}
\Delta(s)=\Delta(\tfrac13u_1^3-u_1u_2^2+sp)
\end{equation*}
is a fourth degree polynomial in \( s \) with negative leading coefficient, so it cannot be positive for all \( s \).
\end{proof}

The curve  \( (t_-,t_+)\to\G\times (t_-,t_+) \) given by \( t\mapsto \vartheta(t)= (e,t) \) is a geodesic which is orthogonal to the orbits. In order to obtain non-trivial examples of complete cohomogeneity one metrics, we therefore need to assume that at least one of the boundary points, say, \( t_- \) is finite. In this case, \( \vartheta(t_-) \) belongs to a special orbit which necessarily must be singular. Otherwise, it would be possible to extend the flow past \( t_-\) in contradiction with the maximality assumption.  

The special stabilizer \( \LH \) is a closed subgroup which is determined, at the Lie algebra level, by the null space of the limit metric. In the language of Lemma~\ref{lemma:endpoint}, we need \( g(t) \) to be defined on \( [t_-,t_+) \), say,
\begin{equation*} 
\lim_{t\downarrow t_-} g(t)=\begin{pmatrix} x & y \\ z & w \end{pmatrix}; 
\end{equation*}
\( \lh \) is then the \( \SO(3) \)-invariant subspace of \( \g \) that contains the null space of \( (xe^1+ye^2)^2 + (ze^1+we^2)^2 \). It follows that \( \LH \) is either three-dimensional or \( \G=\LH \). 

As \( \G \) preserves the metric, the exponential map defines a local diffeomorphism
\begin{equation*}
\G\times_{\LH} \mathcal B\to N,
\end{equation*}
where \( \mathcal B\subset V \) is an open ball inside the normal space, \( V \), to the orbit at \( \vartheta(t_-) \). Since we are concerned with smoothness about the special orbit, there is no harm done by working with \( \G\times_{\LH} V \) rather than \( \G\times_{\LH} \mathcal B \). 

Up to considering a connected component, we can  assume that \( \G \) is connected. Up to taking a covering space we may, in fact, assume \( \G \) is simply-connected. In addition, we can assume that \( \LH \) is connected, since the action of \( \LH/\LH_e \) on \( \G\times_{\LH_e} V \) is free and properly discontinuous, making \( \G\times_{\LH_e} V\to \G\times_{\LH} V \) into a covering map. 

In summary, we are left to consider a manifold of the form \( \G\times_{\LH} V \), where \( \rho\colon \LH\to \GL(V) \)
is a sphere transitive representation. In particular, this implies that \( \dim \LH=\dim V-1 \). 

If \( \G=\LH \) the group must act sphere transitively on \( \bR^7 \). For dimensional reasons this implies \( \G \) is diffeomorphic to \( S^6 \), which is absurd. We can therefore focus on the case when \( \LH \) is three-dimensional and \( V \) is a polar four-dimensional representation. Then we have \( V=\bH \) as a representation of \( \LH=\SU(2) \).

\( \LH \) acts on \( \G\times V \) on the right with fundamental vector fields of the form
\begin{equation*} 
A^*_{g,v} = \left.\frac{d}{dt}\right\lvert_{t=0} (g\exp (t A), \rho(-t A) v)= (L_{g*}A, -\rho_{*e}(A)v), \quad  A\in\lh;
\end{equation*}
this action makes  \( \G\times V \) into an \( \LH \) structure on \( \G\times_{\LH} V \). In particular, the metric determines an \( \LH \)-equivariant mapping \( \G\times V\to S^2(\g\slash\lh\oplus V)^*  \). By \( \G \)-invariance it suffices to consider the restriction to \( V\cong\{e\}\times V \), and this gives rise to a map \( V\to S^2(\g\slash\lh\oplus V)^* \).
Similarly the \( \G_2 \)-form determines a smooth map \( V\to \Lambda^3(\g\slash\lh\oplus V)^* \).

It is now a matter of identifying the \( \G_2 \)-metric and \( 3 \)-form as maps defined on \( V\setminus\{0\} \) and to determine whether these  extend to all of \( V \). To this end let us fix a vector \( v\in V \) so as to get a map 
\begin{equation*} 
\G\times \bR\to G\times_{\LH} V, \quad (g,t)\mapsto [g,tv] 
\end{equation*}
which identifies the tangent space \( T_{[g,v]}\G\times_{\LH} V \) with \( \g\oplus\bR \). Explicitly, by choosing an \( \LH \)-invariant decomposition 
\begin{equation*}
\g=\lh\oplus\m, 
\end{equation*}
the map 
\( \g\oplus \bR \to \m\oplus V \)
is given by 
\begin{equation*}
(A,s)\mapsto ([L_{g*}A]_{\m}, sv) = (L_{g*}[A]_{\m}, sv +\rho_{*e}([A]_{\lh})v).
\end{equation*}
In other words, we have:
\begin{equation}
 \label{eqn:dictionary}
\g\ni A\mapsto \begin{cases} (A,0)& A\in\m \\ (0,\rho_{*e}([A])(v)) & A\in \lh\end{cases}, \quad \D{}{t}\mapsto (0,v).
\end{equation}

We are now in a position to prove the main result of this section: the Bryant-Salamon metric is the only complete full holonomy \( \G_2 \)-metric which arises from invariant intrinsic torsion \( \SO(3) \)-structures.

\begin{theorem}
\label{thm:clsf_G2}
There are exactly two complete, simply-connected  Riemannian manifolds with holonomy contained in \( \G_2 \) that are obtained by evolving a half-flat \( \SO(3) \)-structure with invariant intrinsic torsion. These are the flat metric on \( \bR^7 \) and the Bryant-Salamon metric on the spinor bundle over the space form \( S^3 \).

In particular, the Bryant-Salamon metric is the only one with holonomy equal to \( \G_ 2 \).
\end{theorem}

\begin{proof}
A six-manifold with an \( \SO(3) \)-structure that has zero intrinsic torsion is flat and so evolves trivially under the Hitchin flow. This means we obtain the flat metric on a seven-dimensional manifold \( N \). Assuming \( N \) is simply-connected and complete we obtain flat \( \bR^7 \).

Let us next consider the more interesting case of non-zero intrinsic torsion. By Lemma~\ref{lem:alt_CG} the maximal solution, \(s\mapsto p+qs \), to the Hitchin flow is defined on an open interval \( I \subsetneq \bR \). 

If \( I=(s_-,s_+) \) is bounded the fact that \( \Delta(q+sp) \) is a polynomial of fourth degree in \( s \) implies that the integral
\( \int_{s_-}^{s_+} (\tfrac43\Delta)^{-1/6} ds \) is finite. In particular, the orthogonal geodesic has finite length. Completeness then requires that both boundary points \( s_\pm \) corresponds to special orbits. 

By the same argument, if \( I=(s_-,+\infty) \) there is a special orbit at \( s_- \), and similarly for \( (-\infty,s_+) \). 

In conclusion completeness in the non-flat case is only possible if there is at least one special orbit.

Up to \( \Sigma_3 \) action we may  assume \( \det g \) is positive on \( I \) (corresponding to points on the principal orbits).

According to Lemma~\ref{lemma:endpoint} there are now three cases we need to consider. 

\vspace{.5mm}

\eqref{itm1:endpoint} If  \( p=u_1(u_1^2+u_2^2) \) the endpoint is necessarily of the form  \( q=\lambda u_1^3 \). A computation gives 
\begin{equation*} 
\Delta (q+sp)=-4s^3(\lambda+s),
\end{equation*}
so that the maximal interval is either \( s\in (-\lambda,0) \) or \( (0,-\lambda) \), depending on the sign of \( \lambda \). Both \( s=-\lambda \) and \( s=0 \) should correspond to special orbits. However,  \( q-\lambda p = -\lambda u_1u_2^2 \) does not satisfy the conditions of Lemma~\ref{lemma:endpoint}. It follows that no complete metric arises in this case.

\vspace{.5mm}

\eqref{itm2:endpoint} If \( p=u_1u_2^2 \) and \( q=\lambda u_1^3 \), \( \lambda>0 \), we find that
\begin{equation*} 
\Delta (q+sp)=-4\lambda s^3 
\end{equation*} 
so that \( s\in(-\infty,0) \) and
\begin{equation*}
\int_{-\infty}^0 (\tfrac43\Delta)^{-1/6} ds = \int (-3\lambda s^3)^{-1/6} ds =+\infty.
\end{equation*}
Similarly, for \( \lambda<0 \), we find that \( s \) ranges in \( (0,+\infty) \) and
\( \int_0^{+\infty} (\tfrac43\Delta)^{-1/6} ds\) is infinite.

Assuming \( \lambda>0 \), we have the decomposition
\begin{equation*}
\lambda u_1^3+su_1u_2^2 = u_1(\sqrt\lambda u_1 + \sqrt{-s} u_2)(\sqrt\lambda u_1 - \sqrt{-s} u_2),
\end{equation*}
which leads to \( x=(3\lambda)^{1/3} \), \( w=\sqrt{-s} (3\lambda)^{-1/6} \) and \( y=z=0 \).
At the boundary point, when \( s=0 \), the metric degenerates to  a multiple of 
\begin{equation*}
(e^1)^2+(e^3)^2+(e^5)^2,
\end{equation*}
so that we have the null space \(\lh=\Span{e_2,e_4,e_6} \).

There are two Lie algebras corresponding to this choice of \( p \); these depend on the choice of formal product as either \( u_1\cdot u_2^2 \) or \( u_2 \cdot u_1u_2 \). The associated structure constants are given by
\begin{equation*}
(-36+45,-46,16+25,-62,-14-23,-24), \quad (0,46,0,62,0,24),
\end{equation*}
respectively.

We can solve explicitly in \( t \) using that
\begin{equation*}
t=\int_0^s (3\lambda)^{-1/3}(-s)^{-1/2} ds = -2(3\lambda)^{-1/3}(-s)^{1/2}.
\end{equation*}
This gives \( s=-\tfrac14t^2(3\lambda)^{2/3} \). The metric can now be described as a map 
\begin{multline*}
\G\times(-\infty,0)\to S^2(\g\oplus\Span{\partial\slash\partial t})^*,\\
(g,t)\mapsto(3\lambda)^{2/3}((e^1)^2+(e^3)^2+(e^5)^2)\\
+\tfrac14t^2(3\lambda)^{1/3}((e^2)^2+(e^4)^2+(e^6)^2)+dt^2, 
\end{multline*}
and we can choose \( \m=\Span{e_1,e_3,e_5} \) as our \( \LH \)-invariant complement of \( \lh =\su(2) \).

We already know that the manifold is of the form \( \G\times_{\SU(2)}\bH \). By fixing coordinates \( x=x_0+x_1i + x_2 j + x_3 k \) on \( \bH \) and identifying \( \lh \) with the imaginary quaternions, the map \eqref{eqn:dictionary} becomes
\begin{equation*}
e^2\mapsto -2 t^{-1}dx_1, \quad e^4\mapsto -2t^{-1} dx_2, \quad e^6\mapsto -2t^{-2}dx_3.
\end{equation*}
The restricted metric \( \bR\to S^2(\m\oplus \bH)^* \) therefore reads
\begin{equation*}
t\mapsto(3\lambda)^{2/3}((e^1)^2+(e^3)^2+(e^5)^2)+(3\lambda)^{1/3}(dx_1^2+dx_2^2+dx_3^2)+dx_0^2,
\end{equation*}
and the equivariant extension \( \bH\to  S^2(\m\oplus \bH)^*\) takes the form
\begin{multline*}
x\mapsto (3\lambda)^{2/3}((e^1)^2+(e^3)^2+(e^5)^2)
+(3\lambda)^{1/3}(dx_0^2+dx_1^2+dx_2^2+dx_3^2) \\
+ \frac{1-(3\lambda)^{1/3}}{4\abs{x}^2}d(\abs{x}^2)\otimes d(\abs{x}^2).
\end{multline*}
Clearly this is only smooth when \( \lambda=1\slash3 \), and in this case we are left with the flat metric.

\vspace{.5mm}

\eqref{itm3:endpoint} The case \( u_1(u_1^2-u_2^2) \) corresponds to three different structures on the same Lie algebra \( \so(3)\oplus\so(3) \); these depend on the formal products
\begin{equation*}
u_1\cdot (u_1^2-u_2^2),\quad  (u_1+u_2)\cdot (u_1^2-u_1u_2), \quad (u_1-u_2)\cdot  (u_1^2+u_1u_2).
\end{equation*}
Explicitly the possibilities are \( (u,\g) \), \( (u \ell ,\g) \) and \( (u\ell^2, \g) \) with
\begin{equation}
 \label{eqn:trialityell}
\ell=\begin{pmatrix} -\frac12  & \frac12\\  -\frac32 & -\frac12\end{pmatrix}
\end{equation}
and \( (u,\g) \) given by 
\begin{equation*}
 (e^{36}+e^{45},e^{35}+e^{46},-e^{16}-e^{25},-e^{15}-e^{26},e^{14}+e^{23},e^{13}+e^{24}).
\end{equation*}

At the boundary point \( q \) has the form 
\begin{equation*}
\lambda u_1^3, \quad \lambda (u_1-u_2)^3,\quad  \lambda (u_1+u_2)^3,\quad \lambda\in\bR.
\end{equation*}
As explained in Remark~\ref{remark:changing},  a curve \( g(t) \) can also be viewed as \( g(t)\ell \) up to changing the initial data from \( (u,\g) \) to \( (u\ell,\g) \). Notice that this change has no not effect on the intrinsic torsion polynomial \( p= u_1(u_1^2-u_2^2) \) but cycles between the three possible boundary points \( q \).

This observation means that we can assume \( q=\lambda u_1^3 \) with \( \lambda\neq0 \) (the case when \( q=0 \), meaning \( \lh=\g \), was ruled out in section~\ref{subsection:cohom1}).

Then, for  \( \lambda>0 \), we can write
\( q(s)=u_1^3(s+\lambda)-su_1 u_2^2, \) which has positive discriminant on \( (0,+\infty) \). We then have
\begin{equation*} 
q(s)= u_1(\sqrt{s+\lambda}\,u_1+\sqrt s \,u_2)(\sqrt{s+\lambda}\,u_1-\sqrt s\, u_2),
\end{equation*}
which gives
\begin{equation*}
g=\diag((3(s+\lambda))^{1/3},(3(s+\lambda))^{-1/6} \sqrt s), \quad \det g = (3(s+\lambda))^{1/6}\sqrt s. 
\end{equation*}
Therefore
\begin{equation*}
\int_0^{+\infty} \frac{ds}{\det g} =\int_0^{+\infty}\tfrac43\Delta^{-1/6}ds= +\infty.
\end{equation*}

We first consider  the case \( (u,\g) \). The metric degenerates to 
\begin{equation*}
\lambda^{2/3}(( e^1)^2 + (e^3)^2+(e^5)^2),
\end{equation*}
showing that the special stabilizer is \( \lh=\Span{e_2,e_4,e_6}\cong\su(2) \).

As a map \( \G\times(0,+\infty)\to  S^2(\g\oplus\Span{\partial\slash\partial t})^* \) the metric is given by
\begin{multline*}
(e,t)\mapsto(3(s(t)+\lambda))^{2/3}(( e^1)^2 + (e^3)^2+(e^5)^2)\\
+(3s(t)(s(t)+\lambda))^{-1/3}(( e^2)^2 + (e^4)^2+(e^6)^2) +dt^2.
\end{multline*}
This metric is not explicit in the coordinate \( t \), and the coordinate \( s \) is not suitable for analyzing the smooth extension, since \( ds=0 \) at \( s=0 \). Instead we change the variable to \( z=s^{1/2} \) which is a proper coordinate on the whole geodesic, since
\( dz = \tfrac12 (3(s+\lambda))^{1/6} dt \). In terms of \( z \) the metric becomes
\begin{multline*}
(3(z^{2}+\lambda))^{2/3}(( e^1)^2+ (e^3)^2+(e^5)^2)\\
+z^2(3(z^{2}+\lambda))^{-1/3}(( e^2)^2 + (e^4)^2+(e^6)^2) +4 (3(z^{2}+\lambda))^{-1/3}dz^2.
\end{multline*}

We fix an \( \LH \)-invariant complement of \( \lh \) in \( \g \): \( \m=\Span{e_1,e_3,e_5} \). This enables us to express the map
\(  (0,+\infty)\to S^2(\m\oplus\bH)^* \) as
\begin{multline}
\label{eqn:bryantsalamon}
z\mapsto 3^{-1/3}\bigl(4(z^{2}+\lambda)^{-1/3}(dx_0^2+dx_1^2 + dx_2^2+dx_3^2) \\
+3(z^{2}+\lambda)^{2/3}(( e^1)^2 + (e^3)^2+(e^5)^2)\bigr).
\end{multline}
Both terms have the same form, namely a strictly positive even function in \( z \) multiplied by an invariant element of \( S^2(\m\oplus\bH)^* \); this shows that the metric extends smoothly.
We therefore find a smooth complete metric for each \( \lambda>0 \). Notice that the parameter can be normalized to \( \lambda=1 \) by rescaling \( z \). These are exactly the Bryant-Salamon metrics appearing in \cite[Case ii, p. 840]{Bryant-S:excep}.

With a similar argument, if \( q=\mu u_1^3 \) with \( \mu\) negative we recover \eqref{eqn:bryantsalamon} with \( \lambda=-\mu \).

The cases \( (u\ell,\g) \) and \( (u\ell^2,\g) \) also give rise to the Bryant-Salamon metric on \( (\SU(2)\times\SU(2))\times_{\SU(2)}\bH \). Indeed, this metric is preserved by the action of \( \SU(2) \) (induced by right multiplication on \( \bH \)) that commutes with the action of \( \SU(2)\times\SU(2) \) (induced by left multiplication in the group itself). This means we obtain three isometric cohomogeneity one actions of \( \SU(2)^2 \); they correspond to its different inclusions in
\( \SU(2)^3 \). A change in the choice of the inclusion determines an automorphism of order three which in concrete terms is given by
\begin{equation}
\label{eq:3sym_aut}
(g,h)\mapsto  (h^{-1},gh^{-1}).
\end{equation}

For left-invariant tensors that are also right-invariant under the diagonal action of \( \SU(2) \), the map \eqref{eq:3sym_aut} induces an identification 
which at the Lie algebra level reads
\begin{equation*}
\ell\colon \su(2)\oplus\su(2)\to \su(2)\oplus\su(2), \quad (A_1,A_2)\mapsto(-A_2,A_1-A_2).
\end{equation*}
Identifying \( \su(2)\oplus\su(2) \) with \( \bR^6 \) through the  frame \( u \), we see that \( \ell \) is represented by the matrix \eqref{eqn:trialityell}.

In summary, the three  metrics arising from the evolution of \( (u,\g) \), \( (u\ell, \g) \) and \( (u\ell^2,\g) \) all correspond to the Bryant-Salamon metric, but they are realized with respect to different choices of cohomogeneity one action.
\end{proof}

\begin{remark}
The group of order three generated by \eqref{eqn:trialityell} is the ``triality'' symmetry mentioned in \cite{Atiyah-W:G2metrics}. By taking an appropriate \( \LU(1) \) quotient, this symmetry relates a deformation and two small resolutions of the same conifold.

This triality stems from the fact that the principal orbit \( \SU(2)\times \SU(2) \) is a three-symmetric space, but the \( \SO(3) \)-structures we are considering are not invariant under this symmetry,  unlike the nearly-K\"ahler metric.
\end{remark}

\begin{remark}
In \cite[Corollary 6.4]{Karigiannis-L:Deform_coni},  Karigiannis and Lotay showed that the Bryant-Salamon metric is the only complete \( \G_2 \)-metric that approaches the cone over \( S^3\times S^3 \) sufficiently fast. In our result, we do not make any assumptions on the asymptotical behaviour of the metric. A priori we could have found complete metrics that were not conical at infinity.

Regarding our implicit cohomogeneity one assumption, this can be relaxed if we ask for a complete holonomy \( \G_2 \)-manifold with a real analytic hypersurface which carries a compatible invariant intrinsic torsion \( \SO(3) \)-structure. Then uniqueness follows by combining Theorem~\ref{thm:citclassifies} and the Cauchy-Kovalevskaya theorem.
\end{remark}

\begin{remark}
In case \eqref{itm1:endpoint} of the above proof, meaning when \( p=u_1(u_1^2+u_2^2) \), we showed that the corresponding metrics are not complete, because they do not extend smoothly at one of the boundary points. In \cite{Belgun:Ontheboundary}, however, it was shown that they are ``half-complete'', in the sense that they extend smoothly at the other boundary.
\end{remark}

\bibliographystyle{plain}

\end{document}